\documentclass{amsart}

\usepackage{enumerate}

\usepackage{stmaryrd}

\usepackage{amssymb}

\usepackage{mathrsfs}

\usepackage{cite}

\usepackage{scalerel,xcolor}
\def\darrow{\mathrel{\ThisStyle{\ooalign{$\SavedStyle\rightarrow$\cr%
  \hfil\textcolor{white}{\rule{2\LMpt}{1\LMex}}\kern2\LMpt\hfil}}}}

\newcommand{\R}{\mathbf{R}}
\newcommand{\Z}{\mathbf{Z}}
\newcommand{\C}{\mathbf{C}}

\newcommand{\bA}{\mathbf{A}}
\newcommand{\bP}{\mathbf{P}}

\newcommand{\cC}{\mathcal{C}}
\newcommand{\cS}{\mathcal{S}}

\newcommand{\cO}{\mathcal{O}}

\newcommand{\cI}{\mathcal{I}}

\newcommand{\cL}{\mathcal{L}}
\newcommand{\cE}{\mathcal{E}}

\newcommand{\sL}{\mathscr{L}}

\newcommand{\sM}{\mathscr{M}}

\newcommand{\bL}{\mathbf{L}}

\newcommand{\bG}{\mathbf{G}}

\newcommand{\bx}{\mathbf{x}}
\newcommand{\bn}{\mathbf{n}}

\newcommand\HH{\mathrm{H}}

\newcommand{\an}{\mathrm{an}}

\newcommand{\dmu}{\mathrm{d}\mu}

\DeclareMathOperator{\pos}{pos}

\DeclareMathOperator{\val}{val}

\DeclareMathOperator{\ord}{ord}

\DeclareMathOperator{\ordjac}{ordjac}

\DeclareMathOperator{\divr}{div}

\DeclareMathOperator{\Star}{Star}

\DeclareMathOperator{\Faces}{Faces}

\DeclareMathOperator{\Pic}{Pic}

\DeclareMathOperator{\Spec}{Spec}

\DeclareMathOperator{\Trop}{Trop}

\DeclareMathOperator{\trop}{trop}
\DeclareMathOperator{\gtrop}{gtrop}

\newtheorem{theorem}{Theorem}[section]
\newtheorem{proposition}[theorem]{Proposition}
\newtheorem{lemma}[theorem]{Lemma}
\newtheorem{corollary}[theorem]{Corollary}
\theoremstyle{definition}
\newtheorem{definition}[theorem]{Definition}
\newtheorem{remark}[theorem]{Remark}
\newtheorem{example}[theorem]{Example}
\newtheorem*{acknowledgements}{Acknowledgements}

\title{Motivic Volumes of Fibers of Tropicalization}
\author{Jeremy Usatine}
\address{Department of Mathematics, Yale University}
\email{jeremy.usatine@yale.edu}

\begin{document}
\maketitle

\begin{abstract}
Let $T$ be an algebraic torus over an algebraically closed field, let $X$ be a smooth closed subvariety of a $T$-toric variety such that $U = X \cap T$ is not empty, and let $\mathscr{L}(X)$ be the arc scheme of $X$. We define a tropicalization map on $\mathscr{L}(X) \setminus \mathscr{L}(X \setminus U)$, the set of arcs of $X$ that do not factor through $X \setminus U$. We show that each fiber of this tropicalization map is a constructible subset of $\mathscr{L}(X)$ and therefore has a motivic volume. We prove that if $U$ has a compactification with simple normal crossing boundary, then the generating function for these motivic volumes is rational, and we express this rational function in terms of certain lattice maps constructed in Hacking, Keel, and Tevelev's theory of geometric tropicalization. We explain how this result, in particular, gives a formula for Denef and Loeser's motivic zeta function of a polynomial. To further understand this formula, we also determine precisely which lattice maps arise in the construction of geometric tropicalization.
\end{abstract}

\section{Introduction and Statements of Main Results}

Let $k$ be an algebraically closed field, let $T$ be an algebraic torus over $k$ with character lattice $M$, and let $X$ be a smooth closed subvariety of a $T$-toric variety such that $U = X \cap T \neq \emptyset$. Let $\sL(X)$ be the arc scheme of $X$. In this paper, we define a tropicalization map 
\[
	\trop: \sL(X) \setminus \sL(X \setminus U) \to M^\vee
\]
on the subset of arcs that do not factor through $X \setminus U$. We show that the fibers of this tropicalization map are constructible subsets of $\sL(X)$, and thus in the sense of Kontsevich's theory of motivic integration \cite{Kontsevich}, each of these fibers $\trop^{-1}(w)$ has a well defined motivic volume $\mu_X(\trop^{-1}(w))$. 

We then prove a formula for these motivic volumes, Theorems \ref{constructibilityformula} and \ref{compatiblemodification}, expressing the multivariable generating function 
\[
	\sum_{w \in M^\vee}\mu_X(\trop^{-1}(w)) \bx^w
\] 
as a rational function in terms of the theory of geometric tropicalization, as introduced by Hacking, Keel, and Tevelev in \cite{HackingKeelTevelev}. With the motivation of further understanding this formula for these motivic volumes, we also study certain properties of geometric tropicalization, proving Theorems \ref{gtropproper} and \ref{immersivegtrop}, which we hope will also be of independent interest in tropical geometry.

Before stating the main results of this paper, we present the following family of examples, which demonstrate that these fibers of tropicalization can be used to compute the motivic zeta function of a polynomial, an invariant introduced by Denef and Loeser in \cite{DenefLoeser}. 

\begin{example}
\label{motivatingexample}
Let $f \in k[x_1, \dots, x_n]$ be a nonzero polynomial. Then by definition, the motivic zeta function of $f$ is the series
\[
	Z_f(s) = \int_{\sL(\bA^n)} s^{\ord_f}\dmu_{\bA^n} = \sum_{\ell =0}^\infty \mu_{\bA^n}(\ord_f^{-1}(\ell))s^\ell \in \sM_{\bA^n} \llbracket s \rrbracket,
\]
where $\sM_{\bA^n}$ is the ring obtained by inverting the class of $\bA^1_{\bA^n}$ in the Grothendieck ring of $\bA^n$-varieties. Set $T = \bG_m^{n+1}$, consider $\bA^{n+1}$ as a $T$-toric variety, set $X \hookrightarrow \bA^{n+1}$ to be the graph of the function $f: \bA^{n} \to \bA^1$, and set $U = X \cap T$. Give $\sM_X$ the topology induced by the dimension filtration. We will see by Theorem \ref{constructibilityformula} that the generating function
\[
	\sum_{w  \in \Z^{n+1}}\mu_X(\trop^{-1}(w)) \bx^w = \sum_{w  \in \Z_{\geq 0}^{n+1}}\mu_X(\trop^{-1}(w)) \bx^w \in \sM_X \llbracket \bx \rrbracket
\] 
is in fact an element of the subring $\sM_X \langle \bx \rangle \subset \sM_X \llbracket \bx \rrbracket$ consisting of series whose coefficients converge to 0 as $w \to \infty$. Thus there is a well defined map
\[
	\varphi: \sM_X \langle \bx \rangle \to \widehat{\sM}_X \llbracket s \rrbracket : \bx^{(w_1, \dots, w_{n+1})} \mapsto s^{w_{n+1}},
\]
where $\widehat{\sM}_X$ is the completion of $\sM_X$. Then because $\sL(X \setminus U)$ is a negligible subset of $\sL(X)$, the image of $Z_f(s)$ in $\widehat{\sM}_X\llbracket s \rrbracket$ is equal to
\[
	\varphi\left( \sum_{w  \in \Z^{n+1}}\mu_X(\trop^{-1}(w)) \bx^w \right).
\]
We will see in Example \ref{motivatingexample2} that, in particular, our main results can be used to give a complete list of candidate poles for $Z_f(s)$ in terms of geometric tropicalization.
\end{example}

A major open problem in this subject is the motivic monodromy conjecture. Introduced by Denef and Loeser in \cite{DenefLoeser}, the motivic monodromy conjecture predicts that, when $k = \C$, the roots of the Bernstein polynomial of $f$ give a complete list of candidate poles for the zeta function $Z_f(s)$. Therefore, there has been much interest in developing methods for computing candidate poles for $Z_f(s)$. One setting in which toric methods have proven fruitful for understanding these zeta functions has been in the case where $f$ is non-degenerate with respect to its Newton polyhedron. In this non-degenerate setting, Denef and Hoornaert gave a formula for the $p$-adic zeta function \cite{DenefHoornaert}, and Bories and Veys \cite{BoriesVeys} and Guibert \cite{Guibert} gave motivic versions of this formula. Bultot and Nicaise also reproved Guibert's formula by proving a formula, in terms of log-smooth models, for a related motivic zeta function \cite{BultotNicaise}. Nicaise, Payne, and Schroeter have also studied, in \cite{NicaisePayne} and \cite{NicaisePayneSchroeter}, connections between tropical geometry and Hrushovski and Kazhdan's theory of motivic integration of semi-algebraic sets.

\subsection{Statements of Main Results}

Throughout this paper, $k$ will be an algebraically closed field, and by variety, we will mean an integral scheme that is finite type and separated over $k$. All toroidal embeddings in this paper will have normal boundary components. Using the terminology in \cite{KKMS}, these are the toroidal embeddings without self-intersection. 

For any finite type $k$-scheme $X$, we let $\sL(X)$ denote the arc scheme of $X$, we let $K_0(\mathbf{Var}_X)$ denote the Grothendieck ring of finite type $X$-schemes, we let $\sM_X$ denote the ring obtained from $K_0(\mathbf{Var}_X)$ by inverting the class of $\bA^1_X$, we let $\bL$ denote the class of $\bA^1_X$ in $\sM_X$, for any finite type $X$-scheme $Y$ we let $[Y]$ denote the class of $Y$ in $\sM_X$, and we endow $\sM_X$ with the topology given by the dimension filtration. If $X$ is a smooth variety, we let $\mu_X$ be the motivic measure that assigns a volume in $\sM_X$ to each constructible subset of $\sL(X)$.

\begin{definition}
Let $T$ be an algebraic torus over $k$ with character lattice $M$ and let $X$ be a smooth closed subvariety of a $T$-toric variety such that $U = X \cap T \neq \emptyset$. Let $x \in \sL(X) \setminus \sL(X \setminus U)$ be a point with residue field $k'$. Then $x$ is an arc $\Spec( k'\llbracket t \rrbracket ) \to X$ whose generic point $\eta: \Spec(k'(\!(t)\!)) \to X$ factors through $U$. Then let $\trop(x) \in M^\vee$ be defined so that for any $u \in M$,
\[
	\langle u, \trop(x) \rangle = \val(\chi^u|_U(\eta)),
\]
where $\chi^u$ is the character on $T$ corresponding to $u$ and $\val: k'(\!( t )\!) \to \Z$ is the valuation given by order of vanishing of $t$. This defines a \emph{tropicalization map}
\[
	\trop: \sL(X) \setminus \sL(X \setminus U) \to M^\vee.
\]
\end{definition}

Before stating our main results, we need to define a technical condition that can be satisfied by a simple normal crossing compactification of a very affine variety.

\begin{definition}
\label{has_gtrop}
Let $T$ be an algebraic torus with character lattice $M$, let $N = M^\vee$, let $U \hookrightarrow T$ be a smooth closed subvariety, let $U \subset X$ be an open immersion into a smooth complete variety $X$ such that $X \setminus U$ is a simple normal crossing divisor, and let $\cC(X \setminus U)$ be the set of irreducible components of $X\setminus U$. For each $\cS \subset \cC(X \setminus U)$, let $\varphi^{\cS}: \Z^{\cS} \to N$ be the map of lattices such that for each $D \in \cS$, the standard basis vector of $\Z^{\cS}$ associated to $D$ is sent to $\val_D|_M \in N$.

\begin{enumerate}[(a)]

\item We say that $U \subset X$ \emph{has immersive geometric tropicalization with respect to $U \hookrightarrow T$} if for each $\cS \subset \cC(X\setminus U)$ such that $\bigcap_{D \in \cS}D \neq \emptyset$, the map $\varphi^{\cS}_\R: \R^\cS \to N_\R$ is injective.

\item Let $\sigma$ be a cone in $N$. Then we say that $U \subset X$ \emph{has $\sigma$-compatible geometric tropicalization with respect to $U \hookrightarrow T$} if for each $\cS \subset \cC(X \setminus U)$ such that $\bigcap_{D \in \cS}D \neq \emptyset$, the cone $(\varphi^\cS_\R)^{-1}(\sigma) \cap \R_{\geq 0}^\cS$ is a face of $\R_{\geq 0}^\cS$.

\item Let $\Delta$ be a fan in $N$. Then we say that $U \subset X$ \emph{has $\Delta$-compatible geometric tropicalization with respect to $U \hookrightarrow T$} if $U \subset X$ has $\sigma$-compatible geometric tropicalization with respect to $U \hookrightarrow T$ for each $\sigma \in \Delta$.

\end{enumerate}
\end{definition}

We will see that when $U \subset X$ has $\Delta$-compatible geometric tropicalization with respect to $U \hookrightarrow T$, the closed immersion $U \hookrightarrow T$ can be extended to a proper map from an open subset of $X$ to the $T$-toric variety defined by $\Delta$. We will use this proper map and the motivic change of variables formula to prove Theorem \ref{constructibilityformula} below which, in particular, gives a formula for the motivic volumes of fibers of tropicalization. As we will describe below, if in addition, $U \subset X$ has immersive geometric tropicalization with respect to $U \hookrightarrow T$, then the combinatorics of this formula will take a simpler form.

\begin{remark}
In the above definition and elsewhere in this paper, we let $\val_D|_M$ denote the map $M \to \Z$ given by pulling back characters to the rational function field of $U$ and then applying the valuation associated to the divisor $D$.
\end{remark}

\begin{remark}
Throughout this paper, by a cone (resp. fan) in a lattice $N$, we mean a rational polyhedral cone (resp. fan) in $N_\R$, and unless explicitly stated otherwise, we always assume cones in $N$ are pointed and that fans in $N$ consist of pointed cones.
\end{remark}

We now state a result, which in particular, gives a formula for motivic volumes of fibers of tropicalization in terms of geometric tropicalization.

\begin{theorem}
\label{constructibilityformula}
Let $T$ be an algebraic torus with character lattice $M$, let $N = M^\vee$, let $\Delta$ be a fan in $N$, let $X$ be a smooth closed subvariety of the $T$-toric variety defined by $\Delta$, and assume $U = X \cap T$ is nonempty. Let $U \subset \widetilde{X}$ be an open immersion into a smooth complete variety $\widetilde{X}$ such that $\widetilde{X} \setminus U$ is a simple normal crossing divisor and such that $U \subset \widetilde{X}$ has $\Delta$-compatible geometric tropicalization with respect to $U \hookrightarrow T$. Let $\cC(\widetilde{X} \setminus U)$ be the set of irreducible components of $\widetilde{X} \setminus U$.

Then there exists an assignment of nonnegative integers $m_D$ to each $D \in \cC(\widetilde{X} \setminus U)$ and an assignment of finite type $X$-schemes $Y_\cS$ to each $\cS \subset \cC(\widetilde{X} \setminus U)$ such that the following holds.

\begin{enumerate}[(a)]

\item For each $\cS \subset \cC(\widetilde{X} \setminus U)$ such that $\bigcap_{D \in \cS} D \neq \emptyset$, the cone 
\[
	\pos(\val_D|_M \, | \, D \in \cS)
\]
is a rational pointed cone in $N_\R$. Furthermore, the rational function
\[
	\prod_{D \in \cS} \frac{ \bL^{-(m_D+1)}\bx^{\val_D|_M}}{1 - \bL^{-(m_D+1)}\bx^{\val_D|_M} }
\]
is a well defined element of the ring $\sM_X\langle \pos(\val_D|_M \, | \, D \in \cS) \cap N \rangle$.

\item For each $\cS \subset \cC(\widetilde{X} \setminus U)$ such that there exists $D \in \cS$ with $\val_D|_M \notin |\Delta|$,
\[
	Y_\cS = \emptyset.
\]

\item For each $w \in N$, the set $\trop^{-1}(w)$ is a constructible subset of $\sL(X)$, and
\[
	\mu_X(\trop^{-1}(w)) = \bL^{-\dim X} \sum_{\substack{\cS \subset \cC(\widetilde{X} \setminus U) \\ \bigcap_{D \in \cS} D \neq \emptyset}} (\bL - 1)^{|\cS|} [Y_\cS] F_\cS(w) \in \sM_X,
\]
where the $F_\cS(w) \in \sM_X$ are such that
\[
	\sum_{w \in N} F_\cS(w) \bx^w =  \prod_{D \in \cS} \frac{ \bL^{-(m_D+1)}\bx^{\val_D|_M}}{1 - \bL^{-(m_D+1)}\bx^{\val_D|_M} },
\]
i.e. the $F_\cS(w)$ are the coefficients of the power series expansion of the rational function in part (a).

\end{enumerate}
\end{theorem}

\begin{remark}
We clarify the notation $\sM_X\langle \pos(\val_D|_M \, | \, D \in \cS) \cap N \rangle$ used above. If $\sigma$ is a rational pointed cone in $N_\R$, then for any ring $A$, there is a well defined power series ring
\[
	A\llbracket \sigma \cap N \rrbracket = \left\{ \sum_{w \in \sigma \cap N} a_w \bx^w  \right\}.
\]
If $A$ has a topology, we let $A \langle \sigma \cap N \rangle$ denote the subring of $A \llbracket \sigma \cap N \rrbracket$ consisting of series whose coefficients $a_w$ converge to 0 as $w \to \infty$.
\end{remark}

\begin{remark}
See Theorem \ref{constructibility} in Section \ref{constructibilityproof} for a more explicit description of the integers $m_D$ and the schemes $Y_\cS$ that appear in the statement of Theorem \ref{constructibilityformula}.
\end{remark}

In the proof of Theorem \ref{constructibility}(\ref{constructibilityrational}), we will give another combinatorial description for the $F_\cS(w)$ that appear in the statement of Theorem \ref{constructibilityformula}. This combinatorial description will be in terms of a sum over points of $\Z_{> 0}^\cS$ that map to $w$. In particular, when $U \subset \widetilde{X}$ has immersive geometric tropicalization with respect to $U \hookrightarrow T$, there will be at most one term in this sum.

Our next result shows that any simple normal crossing compactification can be modified to obtain one that satisfies the compatibility condition in the hypotheses of Theorem \ref{constructibilityformula}. In particular, when $k$ has characteristic 0, we always have a formula of the above form for the volumes of fibers of tropicalization.

\begin{theorem}
\label{compatiblemodification}
Let $T$ be an algebraic torus with co-character lattice $N$, let $\Delta$ be a fan in $N$, and let $U \hookrightarrow T$ be a smooth closed subvariety.

\begin{enumerate}[(a)]

\item \label{compatiblemodificationmain} If $U \subset X$ is an open immersion into a smooth complete variety $X$ such that $X \setminus U$ is a simple normal crossing divisor, then there exists a toroidal modification $(U \subset \widetilde{X})\to (U \subset X)$ such that $\widetilde{X}$ is smooth and complete and such that $U \subset \widetilde{X}$ has $\Delta$-compatible geometric tropicalization with respect to $U \hookrightarrow T$. Furthermore, if $U \subset X$ has immersive geometric tropicalization with respect to $U \hookrightarrow T$, then so does $U \subset \widetilde{X}$.

\item If the characteristic of $k$ is 0, there exists an open immersion $U \subset \widetilde{X}$ into a smooth complete variety $\widetilde{X}$ such that $\widetilde{X} \setminus U$ is a simple normal crossing divisor and such that $U \subset \widetilde{X}$ has $\Delta$-compatible geometric tropicalization with respect to $U \hookrightarrow T$.

\end{enumerate}
\end{theorem}

Equipped with these results, we return to Example \ref{motivatingexample}.

\begin{example}
\label{motivatingexample2}
Suppose that $k$ has characteristic 0, and let $f, T, X, U$ be as in Example \ref{motivatingexample}. Then $\bA^{n+1}$ is the $T$-toric variety defined by the positive orthant $\R_{\geq 0}^{n+1}$. By Theorem \ref{compatiblemodification}, there exists an open immersion $U \subset \widetilde{X}$ into a smooth complete variety $\widetilde{X}$ such that $\widetilde{X} \setminus U$ is a simple normal crossing divisor and such that $U \subset \widetilde{X}$ has $\R_{\geq 0}^{n+1}$-compatible geometric tropicalization with respect to $U \hookrightarrow T$. Then by Theorem \ref{constructibilityformula}, using the notation in the statement of that theorem, the generating function
\[
	\sum_{w  \in \Z^{n+1}}\mu_X(\trop^{-1}(w)) \bx^w
\]
is equal to the rational function
\[
  	\bL^{-\dim X} \sum_{\substack{\cS \subset \cC(\widetilde{X} \setminus U) \\ \bigcap_{D \in \cS} D \neq \emptyset}} (\bL - 1)^{|\cS|} [Y_\cS] \prod_{D \in \cS} \frac{ \bL^{-(m_D+1)}\bx^{\val_D|_M}}{1 - \bL^{-(m_D+1)}\bx^{\val_D|_M} } \in \sM_X \langle \bx \rangle.
\] 
Thus by the discussion in Example \ref{motivatingexample}, the image of the zeta function $Z_f(s)$ in $\widehat{\sM}_X\llbracket s \rrbracket$ is equal to the rational function
\[
	 \bL^{-\dim X} \sum_{\substack{\cS \subset \cC(\widetilde{X} \setminus U) \\ \bigcap_{D \in \cS} D  \neq \emptyset}} (\bL - 1)^{|\cS|} [Y_\cS] \prod_{D \in \cS} \frac{ \bL^{-(m_D+1)}\varphi(\bx^{\val_D|_M})}{1 - \bL^{-(m_D+1)}\varphi(\bx^{\val_D|_M}) } \in \widehat{\sM}_X\llbracket s \rrbracket,
\]
where $\varphi(\bx^{(w_1, \dots, w_{n+1})}) = s^{w_{n+1}}$. Note that to prove Theorem \ref{constructibilityformula} in Section \ref{constructibilityproof}, we will use $\widetilde{X}$ to construct a log resolution of the pair $(X, X \setminus U)$, and the formula above for the image of $Z_f(s)$ is the usual formula one gets for the motivic zeta function from this log resolution.
\end{example}

In the next example, we explain how one can use Tevelev's theory of tropical compactifications \cite{Tevelev} and Theorem \ref{constructibilityformula} to recover a formula for the motivic zeta function of a polynomial that is non-degenerate with respect to its Newton polyhedron.

\begin{example}
Keep all notation as in Example \ref{motivatingexample2}, but also assume that $f$ is non-degenerate with respect to its Newton polyhedron. Then $U$ is sch\"{o}n in $T$. Let $\Delta_1$ be any fan in $\Z^{n+1}$ supported on $\Trop(U)$, let $\Sigma$ be any complete fan in $\Z^{n+1}$ containing $\R_{\geq 0}^{n+1}$, and let $\Delta_2 = \{\sigma \cap \sigma'  \, | \, \sigma \in \Delta_1, \sigma' \in \Sigma\}$ be the common refinement of $\Delta_1$ and $\Sigma$. Let $\Delta$ be a unimodular subdivision of $\Delta_2$, and let $X(\Delta)$ be the $T$-toric variety defined by $\Delta$. Then one can check, using the theory of tropical compactifications, that the closure of $U$ in $X(\Delta)$ is a simple normal crossing compactification of $U$, and one can check that by construction of $\Delta$, this compactification has $\R_{\geq 0}^{n+1}$-compatible geometric tropicalization with respect to $U \hookrightarrow T$. Thus we may choose $\widetilde{X}$ to be the closure of $U$ in $X(\Delta)$. In this special setting, one can explicitly calculate the integers $m_D$ in terms of the Newton polyhedron of $f$ and the lattice point $\val_D|_M$. One can also write each class $(\bL-1)^{|\cS|} [Y_\cS]$ as the class of an initial degeneration of $U$. Having done this, one can write the resulting formula for the image of $Z_f(s)$ in $\widehat{\sM}_X\llbracket s \rrbracket$ in a way that does not depend on the choices of $\Delta_1, \Sigma$, or $\Delta$.
\end{example}

We see that the combinatorics of the formula in Theorem \ref{constructibilityformula} depend on the geometric tropicalization maps denoted by $\varphi^\cS$ in Definition \ref{has_gtrop}. We thus study these maps. As a first result in this direction, we classify which maps of lattices can arise as $\varphi^\cS$. We show that up to embedding the codomain into a larger lattice, the maps of the form $\varphi^\cS$ are precisely the lattice maps whose induced map on the positive orthant is proper as a map of topological spaces. This is true even when we restrict ourselves to very affine varieties embedded in their intrinsic torus.

\begin{theorem}
\label{gtropproper}
\begin{enumerate}[(a)]

\item \label{gtropisproper} Let $T$ be an algebraic torus with character lattice $M$, let $N = M^\vee$, let $U \hookrightarrow T$ be a smooth closed subvariety, let $U \subset X$ be an open immersion into a smooth variety $X$ such that $X \setminus U$ is a simple normal crossing divisor, and let $\cC(X \setminus U)$ be the set of irreducible components of $X\setminus U$. For each $\cS \subset \cC(X \setminus U)$, let $\varphi^{\cS}: \Z^{\cS} \to N$ be the map of lattices such that for each $D \in \cS$, the standard basis vector of $\Z^{\cS}$ associated to $D$ is sent to $\val_D|_M \in N$.

If $\cS \subset \cC(X \setminus U)$ is such that $\bigcap_{D \in \cS} D \neq \emptyset$, then the map $\varphi^\cS_\R|_{\R_{\geq 0}^\cS}: \R_{\geq 0}^\cS \to N_\R$ is proper as a map of topological spaces.

\item \label{properisgtrop} Let $L$ be a lattice, let $m \in \Z_{> 0}$, let $\varphi: \Z^m \to L$ be a map of lattices such that
\[
	\varphi_\R|_{\R_{\geq 0}^m}: \R_{\geq 0}^m \to L_\R
\]
is proper as a map of topological spaces, and let $X$ be a smooth projective variety of dimension at least $\max(m,2)$.

Then there exists an open subvariety $U \subset X$, a collection $\{D_1, \dots, D_m\}$ of irreducible components of $X \setminus U$, a closed immersion $U \hookrightarrow T$ into an algebraic torus with co-character lattice $N$, and a map of lattices $\psi: L \to N$ such that
\begin{enumerate}[(i)]

\item $X \setminus U$ is a simple normal crossing divisor,

\item $D_1 \cap \dots \cap D_m \neq \emptyset$,

\item $\psi$ induces an isomorphism of $L$ onto a saturated sublattice of $N$,

\item and $\psi \circ \varphi: \Z^m \to N$ is the map sending the $j$th standard basis vector of $\Z^m$ to $\val_{D_j}|_M$, where $M = N^\vee$.

\end{enumerate}
Furthermore, if $X$ has Picard rank at least $m$, then $U$, $D_1, \dots, D_m$, $U \hookrightarrow T$, and $\psi$ can be chosen so that $M = \cO_U(U)^\times/k^\times$ and $U \hookrightarrow T$ is induced by a section $M \to \cO_U(U)^\times$ of $\cO_U(U)^\times \to M$.

\end{enumerate}
\end{theorem}

\begin{remark}
If $U$ is a very affine variety, the algebraic torus $T$ with character lattice $M = \cO_U(U)^\times / k^\times$, often called the \emph{intrinsic torus} of $U$, and the closed immersions $U \hookrightarrow T$ induced by sections $M \to \cO_U(U)^\times$ of $\cO_U(U)^\times \to M$ are intrinsically defined from the variety $U$. We thus think of Theorem \ref{gtropproper}(\ref{properisgtrop}) as saying something intrinsic about the embedding $U \subset X$. More precisely, Theorem \ref{gtropproper}(\ref{properisgtrop}) describes how much information about boundary divisors can be lost after restricting their valuations to the rational functions that are units on $U$. The fact that $U \hookrightarrow T$ can be chosen to be one of these intrinsic embeddings is also important in the proof of Theorem \ref{immersivegtrop}(\ref{surfacewithoutimmersivegtrop}) below.

In attempting to prove Theorem \ref{gtropproper}(\ref{properisgtrop}), one might begin with a very affine variety $U$ embedded in an algebraic torus $T'$ and then proceed by taking monomial maps from $T'$ to other algebraic tori. But such a method would not allow one to guarantee that the resulting embedding is one of the intrinsic embeddings into $T$, and thus would not say something intrinsic about $U \subset X$.
\end{remark}

Finally, we study when the geometric tropicalization maps are injective, or in the terminology of Definition \ref{has_gtrop}, we study when compactifications have immersive geometric tropicalization. We show that the existence of a compactification with immersive geometric tropicalization is more general than a very affine variety being sch\"{o}n, in the sense of Tevelev \cite{Tevelev}, but also that not all very affine varieties have a compactification with immersive geometric tropicalization.

\begin{theorem}
\label{immersivegtrop}
\begin{enumerate}[(a)]

\item \label{surfacewithoutimmersivegtrop} There exists a very affine surface $U$ such that for all closed immersions $U \hookrightarrow T$ into an algebraic torus and for all open immersions $U \subset X$ into a smooth complete variety $X$ such that $X \setminus U$ is a simple normal crossing divisor, we have that $U \subset X$ does not have immersive geometric tropicalization with respect to $U \hookrightarrow T$.

\item \label{immersivenotschon} Let $X$ be a smooth projective variety of dimension at least 2. Then there exists a very affine open subvariety $U \subset X$ such that $X \setminus U$ is a simple normal crossing divisor and such that if $T$ is the algebraic torus with character lattice $M = \cO_U(U)^\times/k^\times$ and $U \hookrightarrow T$ is a closed immersion induced by a section $M \to \cO_U(U)^\times$ of $\cO_U(U)^\times \to M$, then $U$ is not sch\"{o}n in $T$, and $U \subset X$ has immersive geometric tropicalization with respect to $U \hookrightarrow T$.

\end{enumerate}
\end{theorem}

\begin{remark}
Let $U$ be a very affine variety of dimension at least 2, and let $U \subset X$ be an open immersion into a smooth variety such that $X \setminus U$ is a simple normal crossing divisor. Given a closed immersion $U$ into an algebraic torus $T$, it is easy to find simple normal crossing compactifications of $U$ that do not have immersive geometric tropicalization with respect to $U \hookrightarrow T$. For example, one may blow-up a point that lies on only one boundary divisor of $X$. But this method would not guarantee that all simple normal crossing compactifications of $U$ do not have immersive geometric tropicalization with respect to $U \hookrightarrow T$. By using Theorem \ref{gtropproper}(\ref{properisgtrop}), we can take a reverse approach, instead constructing the desired very affine variety $U$ by starting with its compactification. In Section \ref{immersivegeometrictropicalizationsection}, we apply this idea to a non-ruled minimal surface to prove Theorem \ref{immersivegtrop}(\ref{surfacewithoutimmersivegtrop}).
\end{remark}


\begin{acknowledgements}
I am grateful to Sam Payne for his feedback throughout this project. I would also like to acknowledge useful discussions with Daniel Corey, Netanel Friedenberg, Paul Hacking, Yuchen Liu, and Dhruv Ranganathan. This work was supported by NSF Grant DMS-1702428 and a Graduate Research Fellowship from the NSF.
\end{acknowledgements}


\section{Preliminaries}

We set notation and recall some facts about motivic integration, the topology of Berkovich analytic spaces and tropicalization, and geometric tropicalization.

\subsection{Motivic Integration}

For further detail on the motivic integration results in this sub-section, we refer to the book \cite{CNS}.

For any finite type $k$-scheme $X$, the Grothendieck ring of finite type $X$-schemes will be denoted by $K_0(\mathbf{Var}_X)$, the ring obtained by inverting the Lefschetz class in $K_0(\mathbf{Var}_X)$ will be denoted by $\sM_X$, the image of the Lefschetz class in $\sM_X$ will be denoted by $\bL$, and if $Y$ is a finite type $X$-scheme, the class of $Y$ in $\sM_X$ will be denoted by $[Y]$. We will endow $\sM_X$ with the topology induced by the dimension filtration. If $g: X' \to X$ is a morphism of finite type $k$-schemes, there is ring homomorphism $g^*: \sM_X \to \sM_{X'}$, given by base change, and there is an additive group homomorphism $g_!: \sM_{X'} \to \sM_X$, given by composition with $g$, that satisfies a projection formula with $g^*$.

For any finite type $k$-scheme $X$, the $n$th jet scheme of $X$ will be denoted by $\sL_n(X)$, the truncation morphisms will be denoted by $\theta^m_n: \sL_m(X) \to \sL_n(X)$, the arc scheme of $X$ will be denoted by $\sL(X) = \varprojlim_n \sL_n(X)$, and the morphisms $\sL(X) \to \sL_n(X)$ will be denoted by $\theta_n$. The following theorem is a direct consequence of a more general result, due to Bhatt \cite[Theorem 1.1]{Bhatt}, on points of $X$ valued over ideal-adically complete rings.

\begin{theorem}[Bhatt]
Let $X$ be a finite type $k$-scheme. Then $\sL(X)$ represents the functor taking any $k$-algebra $A$ to the set of $k$-morphisms from $\Spec(A\llbracket t \rrbracket)$ to $X$, and under this identification, each morphism $\theta_n: \sL(X) \to \sL_n(X)$ is the truncation morphism.
\end{theorem}

If $X$ is a smooth variety, $\mu_X$ will denote its associated motivic measure, which assigns to each constructible subset of $\sL(X)$ a volume in $\sM_X$. Note that we are using the notion of constructible subsets of a not necessarily noetherian scheme. In the case of the scheme $\sL(X)$, constructible subsets coincide with the subsets of the form $\theta_n^{-1}(C)$ for some constructible subset $C \subset \sL_n(X)$. For this reason, the constructible subsets of $\sL(X)$ are also called \emph{cylinders}.

\begin{definition}
Let $C$ be a constructible subset of $\sL(X)$ and $\alpha: C \to \Z$ be a function such that for each $n \in \Z$, the fiber $\alpha^{-1}(n)$ is a constructible subset of $C$. Then the \emph{motivic integral of $\alpha$} is
\[
	\int_C \bL^{-\alpha} \dmu_X = \sum_{n \in \Z} \mu_X(\alpha^{-1}(n))\bL^{-n} \in \sM_X.
\]
Note that, by the quasi-compactness of the constructible topology, $\alpha$ only takes finitely many values, so the above sum is finite.
\end{definition}

If $g: X' \to X$ is a morphism of smooth varieties, the jacobian ideal of $g$ is a locally principal ideal on $X'$, and its associated order function on the arc scheme of $X'$ will be denoted by $\ordjac_g: \sL(X') \to \Z_{\geq 0} \cup \{\infty\}$. We will use the following version of the motivic change of variables formula, which is a special case of Theorem 1.2.5 in Chapter 5 of \cite{CNS}.

\begin{theorem}[Motivic Change of Variables Formula]
Let $g: X' \to X$ be a morphism of smooth varieties, let $C$ be a constructible subset of $\sL(X)$ and $C'$ be a constructible subset of $\sL(X')$, and let $\alpha: C \to \Z$ be a function such that each $\alpha^{-1}(n)$ is a constructible subset of $C$. Assume that for each field extension $k'$ of $k$, $\sL(g)$ induces a bijection between the $k'$ points of $C'$ and the $k'$ points of $C$, and assume that $C' \cap \ordjac_g^{-1}(\infty) = \emptyset$. Then
\begin{itemize}

\item the function $\beta = \alpha \circ \sL(g) + \ordjac_g: C' \to \Z$ is a function such that each $\beta^{-1}(n)$ is a constructible subset of $C'$,

\item and
\[
	\int_C \bL^{-\alpha} \dmu_X = g_! \int_{C'} (g^*\bL)^{-\beta}\dmu_{X'} \in \sM_{X}.
\]

\end{itemize}
\end{theorem}

We will also use the following well known elementary fact about motivic volumes, see for example Lemma 1.2.2 in Chapter 6 of \cite{CNS}.

\begin{proposition}
\label{stratumvolume}
Let $X$ be a smooth variety, let $U \subset X$ be an open subvariety such that $X \setminus U$ is a simple normal crossing divisor, let $\cC(X \setminus U)$ be the collection of irreducible components of $X \setminus U$, and for each $D \in \cC(X \setminus U)$, let $\cI_D$ be the ideal sheaf of $D$ in $X$.

Let $\bn = (n_D)_D \in \Z_{\geq 0}^{\cC(X \setminus U)}$, set
\[
	A = \{x \in \sL(X) \, | \, \text{$\ord_{\cI_{D}}(x) = n_D$ for all $D \in \cC(X \setminus U)$}\} ,
\]
set
\[
	\cS = \{D \in \cC(X \setminus U) \, | \, n_D > 0\},
\]
and set
\[
	Y = \left(\bigcap_{D \in \cS} D\right) \setminus \left(\bigcup_{E \in \cC(X \setminus U) \setminus \cS} E\right)
\]
considered as a scheme over $X$ by inclusion.

Then $A$ is a constructible subset of $\sL(X)$, and
\[
	\mu_X(A) = \bL^{-\dim X} (\bL-1)^{|\cS|} \bL^{-\sum_{D \in \cS} n_D} [Y] \in \sM_X.
\]
\end{proposition}

\subsection{Berkovich Spaces and Tropicalization}

If $X$ is a finite type $k$-scheme, $X^\an$ will denote the underlying topological space of the Berkovich analytification of $X$ over the trivially valued field $k$, as defined in \cite{Berkovich}. As a set, $X^\an$ is the disjoint union, over all points $x' \in X$, of the set of valuations on $k(x')$ that extend the trivial valuation on $k$. For each valuation $x \in X^\an$ of $k(x')$ and $f$ a regular function in a neighborhood of $x'$ in $X$, we will let $\val(f(x)) \in \R \cup \{\infty\}$ denote the value obtained by evaluating the image of $f$ in $k(x')$ at the valuation $x$. The topology on $X^\an$ is the coarsest topology such that for each affine open $U \subset X$ with coordinate ring $A$ and $f \in A$, the set $U^\an$ is open and the function $U^\an \to \R \cup \{\infty\}: x \mapsto \val(f(x))$ is continuous. 

If $R$ is a rank 1 valuation ring, then any map $\Spec(R) \to X$ gives a valuation on $k(x')$, where $x'$ is the image of the generic point of $\Spec(R)$. The subset of $X^\an$ consisting of such valuations, as we vary over valuation rings $R$ and maps $\Spec(R) \to X$, is denoted by $X^\beth$. Thus for each $x \in X^\beth$, we can refer to its generic point and special point.

This construction is functorial, so if $g: X \to Y$ is a morphism of finite type $k$-schemes, we get a continuous map $g^\an: X^\an \to Y^\an$. The next proposition follows from Propositions 3.1.3, 3.1.5, 3.4.1, 3.4.6 and 3.4.7 in \cite{Berkovich}.

\begin{proposition}[Berkovich]
If $g: X \to Y$ is a morphism of finite type $k$-schemes, then $g$ is separated (resp. proper) if and only if $g^\an$ is separated (resp. proper) as a map of topological spaces.
\end{proposition}

Let $U \subset X$ be a toroidal embedding without self-intersection, and let $\cC(X \setminus U)$ be the set of irreducible components of $X \setminus U$. We will consider the stratification of $X$ given by the components of
\[
	\left(\bigcap_{D \in \cS} D \right)\setminus \left(\bigcup_{E \in \cC(X \setminus U) \setminus \cS}E\right)
\]
for each $\cS \subset \cC(X \setminus U)$.
For each stratum $Y$ of the toroidal embedding, let $\Star Y$ denote the open set
\[
	\Star Y = \bigcup_{\text{strata $Z$ such that $Y \subset \overline{Z}$}} Z,
\]
let $M^Y$ be the group of Cartier divisors on $\Star Y$ that are supported on $\Star Y \setminus U$, let $N^Y = (M^Y)^\vee$, and let $\sigma^Y \subset N^Y_\R$ be the cone consisting of elements that are nonnegative on all effective divisors in $M^Y$. Let $\Sigma = (|\Sigma|, (\sigma^Y, M^Y)_Y)$ be the cone complex with integral structure associated to $U \subset X$, as defined in \cite{KKMS}.

For each stratum $Y$, we have a map $\trop_{\Star Y}: U^\an \cap (\Star Y)^\beth \to \sigma^Y$ defined so that for each Cartier divisor $m \in M^Y$ and each $x \in U^\an \cap (\Star Y)^\beth$,
\[
	\langle m, \trop_{\Star Y}(x) \rangle = \val(f(x)),
\]
where $f$ is a local equation for $m$ in a neighborhood of the special point of $x$. It is easy to check that this definition is independent of the choice of local equation $f$. These maps glue to give a continuous, surjective, and proper map
\[
	\trop_X: U^\an \cap X^\beth \to |\Sigma|.
\]

Now let $T$ be an algebraic torus with co-character lattice $N$. We also have a continuous, surjective, and proper tropicalization map 
\[
	\trop: T^\an \to N_\R,
\]
defined so that for all $u \in M$ and $x \in T^\an$,
\[
	\langle u, \trop(x) \rangle = \val(\chi^u(x)).
\]

\subsection{Geometric Tropicalization}

Let $U \subset X$ be a toroidal embedding without self-intersection, and let $\Sigma = (|\Sigma|, (\sigma^Y, M^Y)_Y)$ be the cone complex with integral structure associated to $U \subset X$, as above.

Let $T$ be an algebraic torus with character lattice $M$, let $N = M^\vee$, and let $g: U \to T$ be a morphism. For each stratum $Y$ of the toroidal embedding, there is a morphism of lattices $M \to M^Y: u \mapsto \divr(g^*(\chi^u))$. This gives a map $\sigma^Y \to N_\R$, and we glue these maps to get a geometric tropicalization map
\[
	\gtrop: |\Sigma| \to N_\R.
\]
It is clear from the definitions that
\[
	\trop \circ g^\an|_{U^\an \cap X^\beth} = \gtrop \circ \trop_X: U^\an \cap X^\beth \to N_\R.
\]
We will use the following elementary proposition.
\begin{proposition}
\label{smoothgtrop}
Suppose that $X$ is smooth and that $g: U \to T$ is a closed immersion. Let $\cC(X \setminus U)$ be the set of irreducible components of $X \setminus U$, let $\cS \subset \cC(X \setminus U)$, let
\[
	Y_\cS = \left(\bigcap_{D \in \cS} D\right) \setminus \left(\bigcup_{E \in \cC(X \setminus U) \setminus \cS} E\right),
\]
and let $\varphi^\cS: \Z^\cS \to N$ be the map of lattices such that for each $D \in \cS$, the standard basis vector of $\Z^\cS$ associated to $D$ is sent to $\val_D|_M \in N$. Then if $\bigcap_{D \in \cS} D \neq \emptyset$,
\begin{enumerate}[(a)]

\item $Y_\cS$ is nonempty,

\item and if $Y$ is a component of $Y_\cS$, there exists an isomorphism $\psi: N^Y \xrightarrow{\sim} \Z^\cS$ identifying $\sigma^Y$ with $\R_{\geq 0}^\cS$ and identifying the map $\gtrop|_{\sigma^Y}$ with $\varphi^\cS_\R|_{\R_{\geq 0}^\cS}$.

\end{enumerate}
\end{proposition}

\begin{proof}
\begin{enumerate}[(a)]

\item Because $\bigcap_{D \in \cS} D \neq \emptyset$ and $X \setminus U$ is a simple normal crossing divisor, $Y_\cS$ is nonempty.

\item Because $X$ is smooth, the set of Cartier divisors $\{D \cap \Star Y \, | \, D \in \cS\}$ forms a basis for $M^Y$. Let $\{v_D \, | \, D \in \cS\}$ be its dual basis. Let $\psi: N^Y \to \Z^\cS$ be the map of lattices sending each $v_D$ to the standard basis vector of $\Z^\cS$ corresponding to $D \in \cS$. Then $\psi$ is an isomorphism, and it identifies $\varphi^\cS_\R|_{\R_{\geq 0}^\cS}$ with the map $\gtrop|_{\sigma^Y}$.

\end{enumerate}
\end{proof}

The notion of geometric tropicalization, in the case where $g$ is a closed immersion and $U \subset X$ is a compactification with simple normal crossing boundary, was introduced by Hacking, Keel, and Tevelev in \cite{HackingKeelTevelev}. In \cite{LuxtonQu}, Luxton and Qu observed that geometric tropicalization could also be defined when $U \subset X$ is a toroidal embedding without self-intersection.


\section{Geometric Tropicalization of Toroidal Embeddings}
\label{gtroptoroidalsection}

Let $U \subset X$ be a toroidal embedding without self-intersection. For each locally closed stratum $Y$ of the toroidal embedding, let $M^Y$ be the group of Cartier divisors on $\Star Y$ that are supported on $\Star Y \setminus U$, let $N^Y = (M^Y)^\vee$, and let $\sigma^Y \subset N^Y_\R$ be the cone consisting of elements that are nonnegative on all effective divisors in $M^Y$. Let $\Sigma = (|\Sigma|, (\sigma^Y, M^Y)_Y)$ be the cone complex with integral structure associated to $U \subset X$.

Let $T$ be an algebraic torus with character lattice $M$, let $N = M^\vee$, and let $g: U \to T$ be a morphism. As discussed in the preliminaries, we have tropicalization maps $\gtrop: |\Sigma| \to N_\R$, $\trop_X: U^\an \cap X^\beth \to |\Sigma|$, and $\trop:T^\an \to N_\R$.

In this section we study how the geometric tropicalization map is related to extending the map $g$ as well as the properness of $g$. In particular we will prove Theorem \ref{gtropproper}(\ref{gtropisproper}).

We first prove the following lemma.

\begin{lemma}
\label{toroidaltroppreimage}
For each stratum $Y$ of $X$, we have
\[
	 \trop_X^{-1}(\sigma^Y) = U^\an \cap (\Star Y)^\beth.
\]
\end{lemma}

\begin{proof}
Suppose that $x \in \trop_X^{-1}(\sigma^Y)$ and suppose that $Z$ is the stratum containing the special point of $x$. Then $x \in (\Star Z)^\beth$ and for each nonzero effective $m \in M^Z$,
\[
	\langle m, \trop_{\Star Z}(x) \rangle > 0,
\]
so $\trop_X(x)$ is in the interior of $\sigma^Z$. Thus $\sigma^Y \cap \sigma^Z$ is a union of faces of $\sigma^Z$ that intersects the interior of $\sigma^Z$, so $\sigma^Z \subset \sigma^Y$. Therefore $\Star Z \subset \Star Y$ and we have that $x \in U^\an \cap (\Star Y)^\beth$.
\end{proof}

We now prove a proposition about extending the map $g$. We will use this proposition in Section \ref{constructibilityproof} to prove Theorem \ref{constructibilityformula}.

\begin{proposition}
\label{gtropindeterminacy}
Let $\Delta$ be a fan in $N_\R$, and let $X(\Delta)$ be its associated $T$-toric variety. Let $\Sigma_\Delta$ be the sub-complex of $\Sigma$ consisting of cones $\sigma^Y$ in $\Sigma$ for which there exists a cone $\sigma \in \Delta$ such that $\gtrop(\sigma^Y) \subset \sigma$, and let $X_\Delta$ be the open subvariety of $X$ given by the inclusion of $\Sigma_\Delta$ into $\Sigma$.

\begin{enumerate}[(a)]

\item \label{gtropindeterminacyextend} The map $g: U \to T$ can be extended to a map $g_\Delta: X_\Delta \to X(\Delta)$.

\item If the map $g_\Delta$ is proper, then $|\Sigma_\Delta| = \gtrop^{-1}(|\Delta|)$.

\item \label{gtropindeterminacyproper} If $X$ is complete and $|\Sigma_\Delta| = \gtrop^{-1}(|\Delta|)$, then the map $g_\Delta$ is proper.

\end{enumerate}

\end{proposition}

\begin{proof}
\begin{enumerate}[(a)]

\item Let $\sigma^Y$ be a cone in $\Sigma_\Delta$, let $\sigma \in \Delta$ such that $\gtrop(\sigma^Y) \subset \sigma$, and let $X(\sigma)$ be its associated affine $T$-toric variety. The condition $\gtrop(\sigma^Y) \subset \sigma$ guarantees that for all $u \in \sigma^\vee \cap M$, the rational function $g^*(\chi^u)$ is regular on $\Star Y$, so the map $g: U \to T$ extends to a map $\Star Y \to X(\sigma) \subset X(\Delta)$. These maps glue to give
\[
	g_\Delta: X_\Delta \to X(\Delta).
\]

\item Suppose that $g_\Delta$ is proper. Let $x \in U^\an \cap X^\beth$ such that $\trop_X(x) \in \gtrop^{-1}(|\Delta|)$. Then $g^\an(x) \in T^\an \cap X(\Delta)^\beth$. Thus because $g_\Delta$ is proper, $x \in U^\an \cap X_\Delta^\beth$. Therefore
\[
	\trop_X(x) \in |\Sigma_\Delta|. 
\]
The surjectivity of $\trop_X$ gives that
\[
	|\Sigma_\Delta| = \gtrop^{-1}(|\Delta|).
\]

\item Suppose that $X$ is complete and $|\Sigma_\Delta| = \gtrop^{-1}(|\Delta|)$. Let $x \in U^\an$ such that $g^\an(x) \in T^\an \cap X(\Delta)^\beth$. Because $X$ is complete, $x \in U^\an \cap X^\beth$. Then
\[
	\gtrop(\trop_X(x)) = \trop(g^\an(x)) \in |\Delta|,
\]
so
\[
	\trop_X(x) \in \gtrop^{-1}(|\Delta|) = |\Sigma_\Delta|.
\]
Then by Lemma \ref{toroidaltroppreimage}, there exists a stratum $Y$ such that $\sigma^Y$ is a cone in $\Sigma_\Delta$ and $x \in U^\an \cap (\Star Y)^\beth$, so
\[
	x \in U^\an \cap X_\Delta^\beth.
\]
Thus $g_\Delta$ is proper by the valuative criterion.

\end{enumerate}
\end{proof}

We now prove the following statement, from which Theorem \ref{gtropproper}(\ref{gtropisproper}) will follow.

\begin{proposition}
\label{generalgtropproper}
\begin{enumerate}[(a)]

\item If $X$ is complete and $\gtrop: |\Sigma| \to N_\R$ is proper as a map of topological spaces, then $g: U \to T$ is proper.

\item \label{generalgtropisproper} If $g: U \to T$ is proper, then $\gtrop: |\Sigma| \to N_\R$ is proper as a map of topological spaces.

\end{enumerate}
\end{proposition}

\begin{proof}
\begin{enumerate}[(a)]

\item Suppose that $X$ is complete and $\gtrop: |\Sigma| \to N_\R$ is proper. Because $X$ is complete, $U^\an \cap X^\beth = U^\an$, so
\[
	\trop \circ g^\an = \gtrop \circ \trop_X: U^\an \to N_\R.
\]
Then because $\trop_X$ and $\gtrop$ are proper, we have that $\trop \circ g^\an$ is proper, so because $\trop$ is separated, we have that $g^\an$ is proper. Therefore $g$ is proper.

\item Suppose that $g: U \to T$ is proper. Then $\trop \circ g^\an$ is proper. Because $U^\an \cap X^\beth$ is closed in $U^\an$, we have that $\gtrop \circ \trop_X = \trop \circ g^\an|_{U^\an \cap X^\beth}$ is proper. Then because $\trop_X$ is surjective, this implies that $\gtrop$ is proper.

\end{enumerate}
\end{proof}

We now prove the following corollary, which is a restatement of Theorem \ref{gtropproper}(\ref{gtropisproper}).

\begin{corollary}
Suppose that $X$ is smooth and that $g: U \to T$ is a closed immersion.

Let $\cS$ be a collection of irreducible components of $X \setminus U$, and let $\varphi^\cS: \Z^\cS \to N$ be the map of lattices such that for each $D \in \cS$, the standard basis vector of $\Z^\cS$ associated to $D$ is sent to $\val_D|_M \in N$.

Then if $\bigcap_{D \in \cS} D \neq \emptyset$, the map $\varphi^\cS_\R|_{\R_{\geq 0}^\cS}: \R_{\geq 0}^\cS \to N_\R$ is proper as a map of topological spaces.
\end{corollary}

\begin{proof}
This follows from Propositions \ref{smoothgtrop} and \ref{generalgtropproper}(\ref{generalgtropisproper}).
\end{proof}

\section{Volumes of Fibers and Geometric Tropicalization}
\label{constructibilityproof}

In this section, we will prove Theorem \ref{constructibilityformula}.

Let $T$ be an algebraic torus with character lattice $M$, let $N = M^\vee$, let $\Delta$ be a fan in $N$, and let $X(\Delta)$ be its associated $T$-toric variety. Let $X \hookrightarrow X(\Delta)$ be a smooth closed subvariety such that $U = X \cap T$ is nonempty, let $U \subset \widetilde{X}$ be an open immersion into a complete smooth variety $\widetilde{X}$ such that $\widetilde{X} \setminus U$ is a simple normal crossing divisor and such that $U \subset \widetilde{X}$ has $\Delta$-compatible geometric tropicalization with respect to $U \hookrightarrow T$. 

Let $\cC(\widetilde{X} \setminus U)$ be the set of irreducible components of $\widetilde{X} \setminus U$, and set
\[
	\widetilde{X}_\Delta = \widetilde{X} \setminus \bigcap_{\sigma \in \Delta} \bigcup_{\substack{D \in \cC(\widetilde{X} \setminus U) \\ \val_D|_M \notin \sigma}} D.
\]

Note that for each $D \in \cC(\widetilde{X} \setminus U)$, we have that $\val_D|_M \in |\Delta|$ if and only if $D \cap \widetilde{X}_\Delta \neq \emptyset$.

We first observe the following.

\begin{proposition}
The inclusion of $U$ into $X$ extends to a proper map $\widetilde{X}_\Delta \to X$.
\end{proposition}

\begin{proof}
This follows from Proposition \ref{gtropindeterminacy} parts (\ref{gtropindeterminacyextend}) and (\ref{gtropindeterminacyproper}) and the proof of Proposition \ref{smoothgtrop}.
\end{proof}

Now for each $D \in \cC(\widetilde{X} \setminus U)$ such that $\val_D|_M \in |\Delta|$, set $m_D \in \Z_{\geq 0}$ to be the multiplicity of the relative jacobian ideal of $\widetilde{X}_\Delta \to X$ at the divisor $D \cap \widetilde{X}_\Delta$. For each $D \in \cC(\widetilde{X} \setminus U)$ such that $\val_D|_M \notin |\Delta|$, set $m_D = 0$.

In the remainder of this section, we will prove the following theorem, which implies Theorem \ref{constructibilityformula}.

\begin{theorem}
\label{constructibility}

\begin{enumerate}[(a)]

\item \label{constructibilityrational} For each $\cS \subset \cC(\widetilde{X} \setminus U)$ such that $\bigcap_{D \in \cS} D \neq \emptyset$, the cone 
\[
	\pos(\val_D|_M \, | \, D \in \cS)
\]
is a rational pointed cone in $N_\R$. Furthermore, the rational function
\[
	\prod_{D \in \cS} \frac{ \bL^{-(m_D+1)}\bx^{\val_D|_M}}{1 - \bL^{-(m_D+1)}\bx^{\val_D|_M} }
\]
is a well defined element of the ring $\sM_X\langle \pos(\val_D|_M \, | \, D \in \cS) \cap N \rangle$.

\item \label{fibervolume} For each $w \in N$, the set $\trop^{-1}(w)$ is a constructible subset of $\sL(X)$, and
\[
	\mu_X(\trop^{-1}(w)) = \bL^{-\dim X} \sum_{\substack{\cS \subset \cC(\widetilde{X} \setminus U) \\ \bigcap_{D \in \cS} D \cap \widetilde{X}_\Delta \neq \emptyset}} (\bL - 1)^{|\cS|} [Y_\cS] F_\cS(w) \in \sM_X,
\]
where for each $\cS \subset \cC(\widetilde{X} \setminus U)$ such that $\bigcap_{D \in \cS} D \cap \widetilde{X}_\Delta \neq \emptyset$,
\[
	Y_\cS = \left(\widetilde{X}_\Delta \cap \bigcap_{D \in \cS} D\right) \setminus \left(\bigcup_{E \in \cC(\widetilde{X} \setminus U) \setminus \cS} E\right)
\]
is a scheme over $X$ by restriction of the map $\widetilde{X}_\Delta \to X$, and the $F_\cS(w) \in \sM_X$ are such that
\[
	\sum_{w \in N} F_\cS(w) \bx^w =  \prod_{D \in \cS} \frac{ \bL^{-(m_D+1)}\bx^{\val_D|_M}}{1 - \bL^{-(m_D+1)}\bx^{\val_D|_M} },
\]
i.e. the $F_\cS(w)$ are the coefficients of the power series expansion of the rational function in part (a).

\end{enumerate}
\end{theorem}

We now prove Theorem \ref{constructibility}(\ref{constructibilityrational}).

\begin{proof}[Proof of Theorem \ref{constructibility}(\ref{constructibilityrational})]
Let $\cS \subset \cC(\widetilde{X} \setminus U)$ such that $\bigcap_{D \in \cS} D \neq \emptyset$, and let $\varphi^\cS: \Z^\cS \to N$ be the map of lattices such that for each $D \in \cS$, the standard basis vector of $\Z^\cS$ associated to $D$ is sent to $\val_D|_M \in N$.

By Theorem \ref{gtropproper}(\ref{gtropisproper}), proved previously in Section \ref{gtroptoroidalsection}, the map $\varphi^\cS_\R|_{\R_{\geq 0}^\cS}: \R_{\geq 0}^\cS \to N_\R$ is proper as a map of topological spaces. Therefore the cone 
\[
	\pos(\val_D|_M \, | \, D \in \cS) = \varphi^\cS_\R(\R_{\geq 0}^\cS)
\]
is pointed, and it is clearly rational. Thus we only need to show that the rational function
\[
	\prod_{D \in \cS} \frac{ \bL^{-(m_D+1)}\bx^{\val_D|_M}}{1 - \bL^{-(m_D+1)}\bx^{\val_D|_M} }
\]
is a well defined element of the ring $\sM_X\langle \pos(\val_D|_M \, | \, D \in \cS) \cap N \rangle$.

For each $w \in N$, set
\[
	F_\cS(w) = \sum_{\substack{\bn = (n_D)_{D \in \cS} \in \Z_{> 0}^\cS\\ \varphi^{\cS}(\bn) = w}} \bL^{-\sum_{D \in \cS} n_D(m_D+1)}.
\]
By the properness of $\varphi^\cS_\R|_{\R_{\geq 0}^\cS}$, the sum defining $F_\cS(w)$ is finite, so it is a well defined element of $\sM_X$. It is also straightforward to check that
\[
	\prod_{D \in \cS} \frac{ \bL^{-(m_D+1)}\bx^{\val_D|_M}}{1 - \bL^{-(m_D+1)}\bx^{\val_D|_M} } = \sum_{w \in N} F_\cS(w)\bx^w,
\]
and because each $m_D \geq 0$,
\[
	\sum_{w \in N} F_\cS(w)\bx^w \in \sM_X\langle \pos(\val_D|_M \, | \, D \in \cS) \cap N \rangle.
\]
\end{proof}

Let $\trop: \sL(X) \setminus \sL(X \setminus U) \to N$ be the tropicalization map defined in the introduction. We will now prove the first part of Theorem \ref{constructibility}(\ref{fibervolume}).

\begin{proposition}
For each $w \in N$, the set $\trop^{-1}(w)$ is a constructible subset of $\sL(X)$.
\end{proposition}

\begin{proof}
If $w \notin |\Delta|$, then $\trop^{-1}(w) = \emptyset$. Therefore we may assume there exists $\sigma \in \Delta$ such that $w \in \sigma$. Let $X(\sigma)$ be its associated affine $T$-toric variety, and let $X_\sigma = X \cap X(\sigma)$.

For any $u \in \sigma^\vee \cap M$ and $x \in \trop^{-1}(w)$ with generic point $\eta$,
\[
	\val(\chi^u|_U(\eta)) = \langle u, w \rangle \geq 0,
\]
so
\[
	\trop^{-1}(w) \subset \sL(X_\sigma).
\]
Then it is not difficult to check that if $S \subset M$ is any set of semigroup generators for $\sigma^\vee \cap M$,
\[
	\trop^{-1}(w) = \bigcap_{u \in S} \left[\theta_{\langle u, w \rangle-1}^{-1}\left(\sL_{\langle u, w \rangle-1}(V(\chi^u|_{X_\sigma}))\right) \setminus \theta_{\langle u, w \rangle}^{-1}\left(\sL_{\langle u, w \rangle}(V(\chi^u|_{X_\sigma}))\right)\right],
\]
where each $V(\chi^u|_{X_\sigma})$ is the hypersurface of $X_\sigma$ cut out by $\chi^u|_{X_\sigma}$, and by any possible appearance of $\theta_{-1}^{-1}\left(\sL_{-1}(V(\chi^u|_{X_\sigma}))\right)$, we actually mean $\theta_0^{-1}(X_\sigma)$. Also, each $V(\chi^u|_{X_\sigma}) \to X$ is a closed immersion followed by an open immersion, so for any $n \in \Z_{\geq 0}$, the set $\sL_{n}(V(\chi^u|_{X_\sigma}))$ is a constructible subset of $\sL_n(X)$. Therefore by taking $S$ to be finite, we see that $\trop^{-1}(w)$ is a constructible subset of $\sL(X)$.
\end{proof}

The remainder of this section is dedicated to finishing the proof of Theorem \ref{constructibility}(\ref{fibervolume}).

Let $g: \widetilde{X}_\Delta \to X$ be the proper morphism extending the inclusion of $U$ into $X$. Set
\[
	\cC_\Delta = \{D \in \cC(\widetilde{X} \setminus U) \, | \, D \cap \widetilde{X}_\Delta \neq \emptyset\} \subset \cC(\widetilde{X} \setminus U),
\]
and for each $D \in \cC_\Delta$, let $\cI_{D, \Delta}$ be the ideal sheaf of $D \cap \widetilde{X}_\Delta$ in $\widetilde{X}_\Delta$.

For each $\bn = (n_D)_D \in \Z_{\geq 0}^{\cC_\Delta}$, set
\[
	A_\bn = \{x \in \sL(\widetilde{X}_\Delta) \, | \, \text{$\ord_{\cI_{D, \Delta}}(x) = n_D$ for all $D \in \cC_\Delta$}\} ,
\]
and note that 
\[
	\bigcup_{\bn \in \Z_{\geq 0}^{\cC_\Delta}} A_\bn = \sL(\widetilde{X}_\Delta) \setminus \sL(\widetilde{X}_\Delta \setminus U).
\]

\begin{lemma}
\label{ordjac}
For all $\bn = (n_D)_D \in \Z_{\geq 0}^{\cC_\Delta}$ and $x \in A_\bn$,
\[
	\ordjac_g(x) = \sum_{D \in \cC_\Delta} n_D m_D,
\]
and
\[
	\trop(\sL(g)(x)) = \sum_{D \in \cC_\Delta} n_D \val_D|_M.
\]
\end{lemma}

\begin{proof}
By definition of $A_\bn$ and each $m_D$,
\[
	\ordjac_g(x) = \sum_{D \in \cC_\Delta} n_D m_D.
\]
Now let $\cS = \{D \in \cC_\Delta \, | \, n_D > 0\}$, and let $Y$ be the component of $\left(\widetilde{X}_\Delta \cap \bigcap_{D \in \cS} D\right) \setminus \left(\bigcup_{E \in \cC(\widetilde{X} \setminus U) \setminus \cS} E\right)$ such that $\theta_0(x) \in Y$. Then considering $U \subset \widetilde{X}_\Delta$ as a toroidal embedding, let $M^Y$ be the group of Cartier divisors on $\Star Y$ supported on $\Star Y \setminus U$, let $N^Y = (M^Y)^\vee$ and let $\sigma^Y \subset N^Y_\R$ be the cone of elements that are nonnegative on the effective divisors in $M^Y$. We have a map $\trop_{\Star Y}: \sL(\Star Y) \setminus \sL(\Star Y \setminus U) \to \sigma^Y$ defined so that for each Cartier divisor $m \in M^Y$ and each $y \in \sL(\Star Y) \setminus \sL(\Star Y \setminus U)$,
\[
	\langle m, \trop_{\Star Y}(y) \rangle = \ord_f(y),
\]
where $f$ is a local equation for $m$ in a neighborhood of $\theta_0(y)$. Let $\{v_D\}_{D \in \cS}$ be the basis of $N^Y$ that is dual to the basis $\{D \cap \Star Y\}_{D \in \cS}$ of $M^Y$. Then it is easy to check that
\[
	\gtrop \circ \trop_{\Star Y} = \trop \circ \sL(g)|_{\sL(\Star Y) \setminus \sL(\Star Y \setminus U)},
\]
that
\[
	\trop_{\Star Y}(x) = \sum_{D \in \cS} n_D v_D,
\]
and for $D \in \cS$,
\[
	\gtrop(v_D) = \val_D|_M.
\]
Therefore
\[
	\trop(\sL(g)(x)) = \sum_{D \in \cS} n_D \val_D|_M.
\]
\end{proof}

For each $w \in N$, set
\[
	B_w = \bigcup_{\substack{\bn = (n_D)_D \in \Z_{\geq 0}^{\cC_\Delta} \\ \sum_{D \in \cC_\Delta} n_D \val_D|_M = w}} A_\bn.
\]

\begin{lemma}
\label{constructiblebijection}
For each $w \in N$ and each field extension $k'$ of $k$, the map $\sL(g)$ induces a bijection between the $k'$ points of $B_w$ and the $k'$ points of $\trop^{-1}(w)$.
\end{lemma}

\begin{proof}
Because $g: \widetilde{X}_\Delta \to X$ is proper and restricts to the identity on $U$, we have that $\sL(g)$ induces a bijection between the $k'$ points of $\sL(\widetilde{X}_\Delta) \setminus \sL(\widetilde{X}_\Delta \setminus U)$ and the $k'$ points of $\sL(X) \setminus \sL(X \setminus U)$. Thus we are done by Lemma \ref{ordjac}.
\end{proof}

For any $\cS \subset \cC_\Delta$, identify $\Z^\cS$ with the subgroup of $\Z^{\cC_\Delta}$ generated by the standard basis vectors associated to each $D \in \cS$.

\begin{lemma}
\label{constructibleunion}
For any $w \in N$,
\[
	B_w = \bigcup_{\substack{\cS \subset \cC(\widetilde{X} \setminus U) \\ \bigcap_{D \in \cS} D \cap \widetilde{X}_\Delta \neq \emptyset}}\left( \bigcup_{\substack{\bn = (n_D)_D \in \Z_{>0}^\cS \\ \sum_{D \in \cS} n_D \val_D|_M = w}} A_\bn\right),
\]
and the right hand side is a finite disjoint union.
\end{lemma}

\begin{proof}
Let $\cS \subset \cC_\Delta$. For any $\bn \in \Z_{>0}^\cS$ and any $x \in A_\bn$,
\[
	\theta_0(x) \in \bigcap_{D \in \cS} D \cap \widetilde{X}_\Delta.
\]
Thus if $\bigcap_{D \in \cS} D \cap \widetilde{X}_\Delta = \emptyset$ and $\bn \in \Z_{>0}^\cS$, then $A_\bn = \emptyset$. This proves the desired equality of sets above.

By the definition of the $A_\bn$, the union above is clearly disjoint. It is finite by Theorem \ref{gtropproper}(\ref{gtropisproper}).
\end{proof}

For each $\cS \subset \cC(\widetilde{X} \setminus U)$, set
\[
	Y_\cS^\Delta = \left(\widetilde{X}_\Delta \cap \bigcap_{D \in \cS} D\right) \setminus \left(\bigcup_{E \in \cC(\widetilde{X} \setminus U) \setminus \cS} E\right)
\]
considered as a scheme over $\widetilde{X}_\Delta$ by the inclusion map.

\begin{lemma}
\label{constructiblestrata}
Let $\cS \subset \cC_\Delta$. Then for any $\bn = (n_D)_D \in \Z_{>0}^\cS$, the set $A_\bn$ is a constructible subset of $\sL(\widetilde{X}_\Delta)$ and
\[
	\mu_{\widetilde{X}_\Delta}(A_\bn) = g^*\left(\bL^{-\dim X}(\bL - 1)^{|\cS|}\bL^{-\sum_{D \in \cS}n_D}\right) [Y_\cS^\Delta] \in \sM_{\widetilde{X}_\Delta}.
\]
\end{lemma}

\begin{proof}
This follows from Proposition \ref{stratumvolume}, noting that $g^*(\bL)$ is the class of $\bA^1_{\widetilde{X}_\Delta}$ in $\sM_{\widetilde{X}_\Delta}$ and that $\dim X = \dim \widetilde{X}_\Delta$.
\end{proof}

For each $\cS \subset \cC(\widetilde{X} \setminus U)$ such that $\bigcap_{D \in \cS} D \cap \widetilde{X}_\Delta \neq \emptyset$, set $Y_\cS$ to be $Y_\cS^\Delta$ but considered as a scheme over $X$ using $g: \widetilde{X}_\Delta \to X$, and for each $w \in N$, set
\[
	F_\cS(w) = \sum_{\substack{\bn = (n_D)_D \in \Z_{>0}^\cS \\ \sum_{D \in \cS} n_D \val_D|_M = w}} \bL^{-\sum_{D \in \cS} n_D(m_D+1)}.
\]
Note that as in the proof of Theorem \ref{constructibility}(\ref{constructibilityrational}), the sum defining $F_\cS(w)$ is finite, so it is a well defined element of $\sM_X$. Also
\[
	\sum_{w \in N} F_\cS(w) \bx^w =  \prod_{D \in \cS} \frac{ \bL^{-(m_D+1)}\bx^{\val_D|_M}}{1 - \bL^{-(m_D+1)}\bx^{\val_D|_M} }.
\]
We now complete the proof of Theorem \ref{constructibility}(\ref{fibervolume}).

\begin{proposition}
For each $w \in N$,
\[
	\mu_X(\trop^{-1}(w)) = \bL^{-\dim X} \sum_{\substack{\cS \subset \cC(\widetilde{X} \setminus U) \\ \bigcap_{D \in \cS} D \cap \widetilde{X}_\Delta \neq \emptyset}} (\bL-1)^{|\cS|} [Y_\cS] F_\cS(w) \in \sM_X.
\]
\end{proposition}

\begin{proof}
Note that the sums that will appear in this proof are all finite by Theorem \ref{gtropproper}(\ref{gtropisproper}). Let $w \in N$. By Lemmas \ref{constructibleunion} and \ref{constructiblestrata}, $B_w$ is a constructible subset of $\sL(\widetilde{X}_\Delta)$. Thus by Lemma \ref{constructiblebijection} and the change of variables formula,
\[
	\mu_X(\trop^{-1}(w)) = g_! \int_{B_w} (g^*\bL)^{-\ordjac_g} \dmu_{\widetilde{X}_\Delta}.
\]
By Lemmas \ref{ordjac} and \ref{constructibleunion}
\begin{align*}
	&g_!\int_{B_w} (g^*\bL)^{-\ordjac_g} \dmu_{\widetilde{X}_\Delta} =\\
	  &\sum_{\substack{\cS \subset \cC(\widetilde{X} \setminus U) \\ \bigcap_{D \in \cS} D \cap \widetilde{X}_\Delta \neq \emptyset}}\left( \sum_{\substack{\bn = (n_D)_D \in \Z_{>0}^\cS \\ \sum_{D \in \cS} n_D \val_D|_M = w}} g_!\left((g^*\bL)^{-\sum_{D\in \cS} n_D m_D} \mu_{\widetilde{X}_\Delta}(A_\bn)\right)\right).
\end{align*}
By Lemma \ref{constructiblestrata} and the projection formula, for each $\cS \subset \cC(\widetilde{X} \setminus U)$ such that $\bigcap_{D \in \cS} D \cap \widetilde{X}_\Delta \neq \emptyset$,
\begin{align*}
	&\sum_{\substack{\bn = (n_D)_D \in \Z_{>0}^\cS \\ \sum_{D \in \cS} n_D \val_D|_M = w}}g_!\left((g^*\bL)^{-\sum_{D\in \cS} n_D m_D} \mu_{\widetilde{X}_\Delta}(A_\bn)\right) =\\
	&\sum_{\substack{\bn = (n_D)_D \in \Z_{>0}^\cS \\ \sum_{D \in \cS} n_D \val_D|_M = w}} \bL^{-\dim X} (\bL -1)^{|\cS|}  \bL^{-\sum_{D \in \cS} n_D(m_D+1)} [Y_\cS]
\end{align*}
which by construction is equal to
\[
	\bL^{-\dim X} (\bL-1)^{|\cS|} [Y_\cS] F_\cS(w).
\]
Therefore,
\[
	\mu_X(\trop^{-1}(w)) = \bL^{-\dim X} \sum_{\substack{\cS \subset \cC(\widetilde{X} \setminus U) \\ \bigcap_{D \in \cS} D \cap \widetilde{X}_\Delta \neq \emptyset}} (\bL-1)^{|\cS|} [Y_\cS] F_\cS(w).
\]
\end{proof}

\section{$\Delta$-Compatibility of a Cone Complex}
\label{compatibilityconecomplex}

In this section, we will prove Theorem \ref{compatiblemodification}. We begin by defining a combinatorial analog of Definition \ref{has_gtrop}.

\begin{definition}
Let $\Sigma = (|\Sigma|, (\sigma^\alpha, M^\alpha)_\alpha)$ be a cone complex with integral structure, and for each index $\alpha$, let $N^\alpha = (M^\alpha)^\vee$. Let $N$ be a lattice and let $\varphi: |\Sigma| \to N_\R$ be a map such that for each index $\alpha$, the restriction $\varphi|_{\sigma^\alpha}: \sigma^\alpha \to N_\R$ is induced by a map of lattices $N^\alpha \to N$.

We say that $\Sigma$ \emph{has immersive geometric tropicalization with respect to} $\varphi$ if for each index $\alpha$, the map $\varphi|_{\sigma^\alpha}: \sigma^\alpha \to N_\R$ is injective.

Let $\sigma$ be a cone in $N$. Then we say $\Sigma$ \emph{has $\sigma$-compatible geometric tropicalization with respect to} $\varphi$ if for each index $\alpha$, the cone $\varphi^{-1}(\sigma) \cap \sigma^\alpha$ is a face of $\sigma^\alpha$.

Let $\Delta$ be a fan in $N$. Then we say that $\Sigma$ \emph{has $\Delta$-compatible geometric tropicalization with respect to} $\varphi$ if $\Sigma$ has $\sigma$-compatible geometric tropicalization with respect to $\varphi$ for each $\sigma \in \Delta$.
\end{definition}

\begin{remark}
\label{compatiblecomplexcompatiblecompactification}
Let $T$ be an algebraic torus with co-character lattice $N$, let $\Delta$ be a fan in $N$, let $U \hookrightarrow T$ be a smooth closed subvariety, let $U \subset X$ be an open immersion into a smooth complete variety $X$ such that $X \setminus U$ is a simple normal crossing divisor, and let $\Sigma$ be the cone complex with integral structure associated to the toroidal embedding $U \subset X$.

Then by Proposition \ref{smoothgtrop}, $U \subset X$ has immersive geometric tropicalization with respect to $U \hookrightarrow T$ if and only if $\Sigma$ has immersive geometric tropicalization with respect to $\gtrop$. Similarly, $U \subset X$ has $\Delta$-compatible geometric tropicalization with respect to $U \hookrightarrow T$ if and only if $\Sigma$ has $\Delta$-compatible geometric tropicalization with respect to $\gtrop$.
\end{remark}

\subsection{Some Lemmas on Polyhedral Cones}

In this subsection, we will prove some lemmas that will be used in the remainder of Section \ref{compatibilityconecomplex}. Let $N$ be a lattice. Note that we will use $\preceq$ to denote the partial relation of inclusion of a face into a cone, and we will use $\wedge$ to denote the operation of taking common refinement of fans. 

We first recall the following well known result, see for example \cite[III. Theorem 2.8]{Ewald}.

\begin{theorem}
Let $\Delta$ be a fan in $N$. Then there exists a fan $\overline{\Delta}$ in $N$ such that $|\overline{\Delta}| = N_\R$ and $\Delta \subset \overline{\Delta}$.
\end{theorem}

We will now prove two elementary lemmas.

\begin{lemma}
\label{subconeface}
Let $\gamma, \gamma'$ be not necessarily pointed cones in $N$ with $\gamma' \subset \gamma$, and let $\nu \preceq \gamma$. Then
\[
	\gamma' \cap \nu \preceq \gamma'.
\]
\end{lemma}

\begin{proof}
Let $u \in \gamma^\vee$ be such that $\nu = \gamma \cap u^\bot$. Then because $u \in (\gamma')^\vee$,
\[
	\gamma' \cap \nu = \gamma' \cap (\gamma \cap u^\bot) = \gamma' \cap u^\bot \preceq \gamma'.
\]
\end{proof}

\begin{lemma}
\label{convexunionfaces}
Let $\gamma$ be a cone in $N$ and $\nu_1, \dots, \nu_r \preceq \gamma$ such that $\nu = \bigcup_{i=1}^r \nu_i$ is a polyhedral cone in $N_\R$. Then
\[
	\nu \preceq \gamma.
\]
\end{lemma}

\begin{proof}
Let $v_1, v_2 \in \gamma$. Because $\nu$ is a polyhedral cone in $N_\R$, we only need to show that if $v_1+v_2 \in \nu$ then $v_1, v_2 \in \nu$. This holds because $\nu$ is a union of faces of $\gamma$.
\end{proof}

For the remainder of this subsection, let $\sigma$ be a cone in $N$, let $\tau$ be a not necessarily pointed cone in $N$, and let $\Delta$ be a fan in $N$ such that $|\Delta| \subset \sigma$ and such that for each $\eta \in \Delta$,
\[
	\eta \cap \tau \preceq \eta.
\]

We will devote the remainder of this subsection to proving the following technical lemma.
\begin{lemma}
\label{technical}
There exists a fan $\Sigma$ in $N$ such that $|\Sigma| = \sigma$, $\Delta \subset \Sigma$, and for each $\sigma' \in \Sigma$,
\[
	\sigma' \cap \tau \preceq \sigma'.
\]
\end{lemma}

First we fix some notation. Fix a fan $\overline{\Delta}$ in $N$ such that $|\overline{\Delta}| = N_\R$ and $\Delta \subset \overline{\Delta}$. Let $\Theta = \overline{\Delta} \wedge \Faces(\tau)$ be the common refinement of $\overline{\Delta}$ and $\Faces(\tau)$. Note that because each cone in $\overline{\Delta}$ is pointed, so is each cone in $\Theta$. Thus $\Theta$ is a fan in $N$ with $|\Theta| = \tau$.

\begin{lemma}
The collection $\Theta \cup \Delta$ is a fan in $N$.
\end{lemma}

\begin{proof}
Let $\gamma \in \overline{\Delta}$, $\gamma' \preceq \tau$, and $\eta \in \Delta$. We need to show that $(\gamma \cap \gamma') \cap \eta$ is a common face of $\gamma \cap \gamma'$ and $\eta$.

By Lemma \ref{subconeface},
\[
	\gamma' \cap \eta = (\eta \cap \tau) \cap \gamma' \preceq \eta \cap \tau.
\]
Thus by the hypotheses on $\Delta$,
\[
	\gamma' \cap \eta \preceq \eta \cap \tau \preceq \eta,
\]
so $\gamma' \cap \eta \in \Delta \subset \overline{\Delta}$. Thus
\[
	(\gamma \cap \gamma') \cap \eta  \preceq \gamma' \cap \eta \preceq \eta.
\]
Similarly, $(\gamma \cap \gamma') \cap \eta \preceq \gamma$, so by Lemma \ref{subconeface},
\[
	(\gamma \cap \gamma') \cap \eta = (\gamma \cap \gamma') \cap ((\gamma \cap \gamma') \cap \eta) \preceq \gamma \cap \gamma',
\]
and we are done.
\end{proof}

Now fix a fan $\overline{\Theta}$ in $N$ such that $|\overline{\Theta}| = N_\R$ and $\Theta \cup \Delta \subset \overline{\Theta}$. Set $\widetilde{\Delta} = \overline{\Delta} \wedge \Faces(\sigma)$ and $\Sigma = \overline{\Theta} \wedge \widetilde{\Delta}$.

\begin{lemma}
\label{subdivisionintersection}
For any $\sigma' \in \Sigma$ and $\nu \in \Theta$,
\[
	\sigma' \cap \nu \preceq \sigma'.
\]
\end{lemma}

\begin{proof}
Let $\gamma \in \overline{\Theta}, \gamma' \in \widetilde{\Delta}$ be such that $\sigma' = \gamma \cap \gamma'$. Then $\gamma \cap \nu \in \overline{\Theta}$ so
\[
	\sigma' \cap \nu = (\gamma \cap \gamma') \cap \nu = (\gamma \cap \nu) \cap \gamma' \in \Sigma.
\]
Thus
\[
	\sigma' \cap \nu = (\sigma' \cap \nu) \cap \sigma' \preceq \sigma'.
\]
\end{proof}

We now finish the proof of Lemma \ref{technical}.

\begin{proof}
By construction
\[
	|\Sigma| = \sigma.
\]
Let $\eta \in \Delta$. Then $\eta = \eta \cap \sigma \in \widetilde{\Delta}$ and $\eta \in \overline{\Theta}$, so
\[
	\eta = \eta \cap \eta \in \Sigma.
\]
Therefore,
\[
	\Delta \subset \Sigma.
\]
Let $\sigma' \in \Sigma$. Then by Lemmas \ref{convexunionfaces} and \ref{subdivisionintersection},
\[
	\sigma' \cap \tau = \bigcup_{\nu \in \Theta} \sigma' \cap \nu \preceq \sigma'.
\]
\end{proof}

\subsection{$\Delta$-Compatible Subdivision}

Let $\Sigma = (|\Sigma|, (\sigma^\alpha, M^\alpha)_\alpha)$ be a cone complex with integral structure, and for each index $\alpha$, let $N^\alpha = (M^\alpha)^\vee$. Let $N$ be a lattice, let $\Delta$ be a fan in $N$, and let $\varphi: |\Sigma| \to N_\R$ be a map such that for each index $\alpha$, the restriction $\varphi|_{\sigma^\alpha}: \sigma^\alpha \to N_\R$ is induced by a map of lattices $N^\alpha \to N$.

\begin{lemma}
\label{keepsubdividing}
If $\sigma$ is a cone in $N$, $\Sigma'$ is a subdivision of $\Sigma$, and $\Sigma$ has $\sigma$-compatible geometric tropicalization with respect to $\varphi$, then $\Sigma'$ has $\sigma$-compatible geometric tropicalization with respect to $\varphi$.
\end{lemma}

\begin{proof}
This follows from Lemma \ref{subconeface}.
\end{proof}

We now prove the combinatorial analog of Theorem \ref{compatiblemodification}.

\begin{proposition}
\label{combinatorialcompatiblemodification}
There exists a unimodular subdivision $\widetilde{\Sigma}$ of $\Sigma$ such that $\widetilde{\Sigma}$ has $\Delta$-compatible geometric tropicalization with respect to $\varphi$.
\end{proposition}

\begin{proof}
Let $\sigma \in \Delta$. We will first show that there exists a subdivision $\Sigma'$ of $\Sigma$ such that $\Sigma'$ has $\sigma$-compatible geometric tropicalization with respect to $\varphi$. 

Assume that there exists a subdivision $\Sigma'_{\ell-1} = (|\Sigma'_{\ell-1}|, (\eta^\beta, L^\beta)_\beta)$ of the $\ell-1$ skeleton of $\Sigma$ such that for each index $\beta$, the cone $\varphi^{-1}(\sigma) \cap \eta^\beta$ is a face of $\eta^\beta$. Now suppose that $\sigma^\alpha$ is a cone in $\Sigma$ that is $\ell$ dimensional. Then by Lemma \ref{convexunionfaces}, the restriction $\Sigma'_{\ell-1}|_{\sigma^\alpha}$ is a fan in $N^\alpha$. Then by Lemma \ref{technical}, there exists a fan $\Sigma^\alpha$ in $N^\alpha$ such that $|\Sigma^\alpha| = \sigma^\alpha$, $\Sigma'_{\ell-1}|_{\sigma^\alpha} \subset \Sigma^\alpha$, and for each $\sigma' \in \Sigma^\alpha$, the cone $\varphi^{-1}(\sigma) \cap \sigma'$ is a face of $\sigma'$. Repeating this for all $\ell$ dimensional cones $\sigma^\alpha$ in $\Sigma$, we see there exists a subdivision $\Sigma'_\ell = (|\Sigma'_{\ell}|, (\eta^\beta, L^\beta)_\beta)$ of the $\ell$ skeleton of $\Sigma$ such that for each index $\beta$, the cone $\varphi^{-1}(\sigma) \cap \eta^\beta$ is a face of $\eta^\beta$. Thus by induction on $\ell$, there exists a subdivision $\Sigma'$ of $\Sigma$ such that $\Sigma'$ has $\sigma$-compatible geometric tropicalization with respect to $\varphi$.

Now by Lemma \ref{keepsubdividing}, repeating the above for all $\sigma \in \Delta$, there exists a subdivision $\Sigma''$ of $\Sigma$ such that $\Sigma''$ has $\Delta$-compatible geometric tropicalization with respect to $\varphi$. Now again by Lemma \ref{keepsubdividing}, using toroidal resolution of singularities, there exists a unimodular subdivision $\widetilde{\Sigma}$ of $\Sigma$ such that $\widetilde{\Sigma}$ has $\Delta$-compatible geometric tropicalization with respect to $\varphi$.
\end{proof}

\subsection{$\Delta$-Compatible Modification}

We now complete the proof of Theorem \ref{compatiblemodification}.

\begin{proof}
\begin{enumerate}[(a)]

\item This follows from Remark \ref{compatiblecomplexcompatiblecompactification} and Proposition \ref{combinatorialcompatiblemodification}.

\item By resolution of singularities, there exists an open immersion $U \subset X$ into a smooth complete variety $X$ such that $X \setminus U$ is a simple normal crossing divisor. Therefore the result follows from part (\ref{compatiblemodificationmain}).

\end{enumerate}
\end{proof}

\section{Geometric Tropicalization with a Given Map of Lattices}

In this section, we will prove Theorem \ref{gtropproper}(\ref{properisgtrop}).

Let $L$ be a lattice, let $m \in \Z_{>0}$, let $\varphi: \Z^m \to L$ be such that
\[
	\varphi_\R|_{\R_{\geq 0}^m}: \R_{\geq 0}^m \to L_\R
\]
is proper as a map of topological spaces, and let $X$ be a smooth projective variety of dimension at least $\max(m,2)$.

Set $L^\varphi = \varphi_\R(\R^m) \cap L \subset L$.

\begin{lemma}
$\varphi_\R(\R_{\geq 0}^m)$ is contained in a unimodular cone in $L^\varphi$.
\end{lemma}

\begin{proof}
Because $\varphi_\R|_{\R_{\geq 0}^m}$ is proper, $\varphi_\R(\R_{\geq 0}^m)$ is a pointed cone in $L^\varphi_\R$ and therefore is contained in a unimodular cone in $L^\varphi$.
\end{proof}

Now let $v_1, \dots, v_n$ be a basis for $L^\varphi$ such that $\varphi_\R(\R_{\geq 0}^m) \subset \pos(v_1, \dots, v_n)$, and let $v_1, \dots, v_n, v'_1, \dots, v'_{n'}$ be a basis for $L$.

Let $e_1, \dots, e_m$ be the standard basis for $\Z^m$, and let $(a_{ij})_{i \in \{1, \dots, n\}, j \in \{1, \dots, m\}}$ be such that for each $j \in \{1, \dots, m\}$,
\[
	\varphi(e_j) = \sum_{i = 1}^n a_{ij}v_i.
\]

\begin{lemma}
\label{matrixentries}

\begin{enumerate}[(a)]

\item \label{matrixentriesnonnegative} For each $i \in \{1, \dots, n\}, j \in \{1, \dots, m\}$, we have $a_{ij} \in \Z_{\geq 0}$.

\item \label{matrixentriesnonvanishing} For each $j \in \{1, \dots, m\}$, there exists some $i \in \{1, \dots, n\}$ such that $a_{ij} \neq 0$.

\item \label{matrixentrieslinearindependence} The matrix $(a_{ij})_{ij}$ has linearly independent rows.

\end{enumerate}
\end{lemma}

\begin{proof}
\begin{enumerate}[(a)]

\item This follows from the fact that $\varphi_\R(\R_{\geq 0}^m) \subset \pos(v_1, \dots, v_n)$.

\item This follows from the properness of $\varphi_\R|_{\R_{\geq 0}^m}$, which in particular implies that each $\varphi(e_j) \neq 0$.

\item This follows from the fact that $\varphi_\R$ surjects onto $L^\varphi_\R$.

\end{enumerate}
\end{proof}

Let $\cL'_1, \dots, \cL'_m$ be very ample line bundles on $X$ such that for all $\ell' \in \Z_{>0}$, the line bundles $(\cL'_1)^{\otimes \ell'}, \dots, (\cL'_m)^{\otimes \ell'}$ are distinct from each other.

\begin{remark}
For example, if $\cL$ is a very ample line bundle, for each $j \in \{1, \dots, m\}$, we can set $\cL'_j = \cL^{\otimes j}$, or if the Picard rank of $X$ is at least $m$, we can let $\cL'_1, \dots, \cL'_m$ be very ample line bundles that are independent in $\Pic(X)$.
\end{remark}

\begin{lemma}
\label{power}
There exists $\ell' \in \Z_{>0}$ such that
\[
	\dim \HH^0 \left( \bigotimes_{j =1}^m ((\cL'_j)^{\otimes \ell'})^{\otimes a_{1j}} \right) - 2 \geq n'. 
\]
\end{lemma}

\begin{proof}
By Lemma \ref{matrixentries}(\ref{matrixentrieslinearindependence}), there exists some $j \in \{1, \dots, m\}$ such that $a_{1j} \neq 0$. Thus by Lemma \ref{matrixentries}(\ref{matrixentriesnonnegative}), the line bundle $\cE'_1 = \bigotimes_{j=1}^m (\cL'_j)^{\otimes a_{1j}}$ is very ample, and thus there exists $\ell' \in \Z_{>0}$ such that 
\[
	\dim \HH^0 \left( \bigotimes_{j =1}^m ((\cL'_j)^{\otimes \ell'})^{\otimes a_{1j}} \right) - 2 = \dim\HH^0((\cE'_1)^{\otimes \ell'}) - 2 \geq n'.
\]
\end{proof}

Now let $\ell'$ be as in Lemma \ref{power}, and for each $j \in \{1, \dots, m\}$, set
\[
	\cL_j = (\cL'_j)^{\otimes \ell'}.
\]
By Bertini's theorem, there exist $f_1, \dots, f_m$ such that for each $j \in \{1, \dots, m\}$, we have that $f_j \in \HH^0(\cL_j)$ and $\divr f_j = D_j$ is smooth and irreducible, and $D_1 + \dots + D_m$ is a simple normal crossing divisor. For each $i \in \{1, \dots, n\}$, set
\[
	\cE_i = \bigotimes_{j=1}^m \cL_j^{\otimes a_{ij}}, \qquad s_i = \bigotimes_{j=1}^m f_j^{\otimes a_{ij}} \in \HH^0(\cE_i), \qquad r_i = \dim \HH^0(\cE_i) - 1.
\]
Note that by our choice of $\ell'$, we have that
\[
	r_1 -1 \geq n'.
\]

\begin{lemma}
For each $i \in \{1, \dots, n\}$, the line bundle $\cE_i$ is very ample.
\end{lemma}

\begin{proof}
For each $i$, by Lemma \ref{matrixentries}(\ref{matrixentriesnonnegative}) and Lemma \ref{matrixentries}(\ref{matrixentrieslinearindependence}), each $a_{ij} \geq 0$ and for some $j \in \{1, \dots, m\}$, we have that $a_{ij} \neq 0$. Thus $\cE_i$ is very ample.
\end{proof}

Now by Bertini's theorem, there exists $(s_i^{(\ell)})_{i \in \{1, \dots, n\}, \ell \in \{1, \dots, r_i\}}$ such that
\begin{itemize}

\item each $s_i^{(\ell)} \in \HH^0(\cE_i)$ and each $\divr s_i^{(\ell)} = E_i^{(\ell)}$ is smooth and irreducible,

\item $D_1 + \dots + D_m + \sum_{i =1}^n \sum_{\ell =1}^{r_i} E_i^{(\ell)}$ is a simple normal crossing divisor,

\item and $s_i, s_i^{(1)}, \dots, s_i^{(r_i)}$ is a basis for $\HH^0(\cE_i)$ for each $i \in \{1, \dots, n\}$.

\end{itemize}

For each $i \in \{1, \dots, n\}$, let $x_i, x_i^{(1)}, \dots, x_i^{(r_i)}$ be homogeneous coordinates for $\bP^{r_i}$ with corresponding hyperplanes $H_i, H_i^{(1)}, \dots, H_i^{(r_i)}$ and let $X \hookrightarrow \bP^{r_i}$ be the closed immersion such that $x_i, x_i^{(1)}, \dots, x_i^{(r_i)}$ pull back to $s_i, s_i^{(1)}, \dots, s_i^{(r_i)}$, respectively. Set
\[
	\bP = \prod_{i=1}^n \bP^{r_i},
\]
and let each $\pi_i: \bP \to \bP^{r_i}$ be the $i$th projection. Let $X \hookrightarrow \bP$ be the closed immersion induced by $\{X \hookrightarrow \bP^{r_i}\}_i$. Set
\[
	T = \bP \setminus \left( \bigcup_{i =1}^n \pi_i^{-1}(H_i) \cup \pi_i^{-1}(H_i^{(1)}) \cup \dots \cup \pi_i^{-1}(H_i^{(r_i)}) \right)
\]
and
\[
	U = X \cap T \subset X.
\]
Let $M$ be the character lattice of the algebraic torus $T$, and let $N = M^\vee$. For each $i \in \{1, \dots, n\}$, let $w_i, w_i^{(1)}, \dots, w_i^{(r_i)} \in N$ be the first lattice points of the rays corresponding to $\pi_i^{-1}(H_i), \pi_i^{-1}(H_i^{(1)}), \dots, \pi_i^{-1}(H_i^{(r_i)})$, respectively.

Let $\psi: L \to N$ be defined by $\psi(v_i) = w_i$ for each $i \in \{1, \dots, n\}$ and $\psi(v'_{i'}) = w_1^{(i')}$ for each $i' \in \{1, \dots, n'\}$.

\begin{lemma}
\label{boundaryintersection}

\begin{enumerate}[(a)]

\item \label{boundarydiv} For each $i \in \{1, \dots, n\}$, $\ell \in \{1, \dots, r_i\}$, we have the scheme theoretic intersections
\[
	X \cap \pi_i^{-1}(H_i) = \divr s_i = \sum_{j =1}^m a_{ij} D_j,
\]
and
\[
	X \cap \pi_i^{-1}(H_i^{(\ell)}) = \divr s_i^{(\ell)} = E_i^{(\ell)}.
\]

\item \label{boundaryval} For each $j \in \{1, \dots, m\}$,
\[
	\val_{D_j}|_M = \sum_{i =1}^n a_{ij} w_i,
\]
and for each $i \in \{1, \dots, n\}, \ell \in \{1, \dots, r_i\}$,
\[
	\val_{E_i^{(\ell)}}|_M = w_i^{(\ell)}.
\]

\end{enumerate}
\end{lemma}

\begin{proof}
This follows from the construction of the map $X \hookrightarrow \bP$.
\end{proof}

We now prove parts (i)-(iv) of Theorem \ref{gtropproper}(\ref{properisgtrop}).

\begin{proof}
\begin{enumerate}[(i)]

\item By Lemmas \ref{matrixentries}(\ref{matrixentriesnonvanishing}) and \ref{boundaryintersection}(\ref{boundarydiv}),
\[
	X \setminus U = D_1 + \dots + D_m + \sum_{i=1}^n \sum_{\ell=1}^{r_i} E_i^{(\ell)}.
\]

\item Because $\cL_1, \dots, \cL_m$ are very ample and $X$ has dimension at least $m$,
\[
	D_1 \cap \dots \cap D_m \neq \emptyset.
\]

\item Because $r_1 -1 \geq n'$, the vectors $w_1, \dots, w_n, w_1^{(1)}, \dots, w_1^{(n')}$ are the first lattice points of the rays generating a cone in the fan defining $\bP$ as a $T$-toric variety. Thus because $\bP$ is smooth, we have that $w_1, \dots, w_n, w_1^{(1)}, \dots, w_1^{(n')}$ can be completed to a basis for $N$, so the result follows.

\item By Lemma \ref{boundaryintersection}(\ref{boundaryval}), for each $j \in \{1, \dots, m\}$,
\begin{align*}
	(\psi \circ \varphi)(e_j) &= \psi\left( \sum_{i=1}^n a_{ij}v_i \right)\\
	&= \sum_{i=1}^n a_{ij}w_i\\
	&= \val_{D_j}|_M.
\end{align*}

\end{enumerate}
\end{proof}

The remainder of Theorem \ref{gtropproper}(\ref{properisgtrop}) follows from the next two propositions.

\begin{proposition}
\label{properisgtropinjective}
The map $M \to \cO_U(U)^\times/k^\times$ induced by $U \hookrightarrow T$ is injective.
\end{proposition}

\begin{proof}
Let $u \in M$ such that $\chi^u|_U \in k^\times$. Considering $\chi^u$ as a rational function on $\bP$, write
\[
	\divr (\chi^u) = \sum_{i =1}^n h_i \pi_i^{-1}(H_i) + \sum_{i = 1}^n \sum_{\ell =1}^{r_i} h_i^{(\ell)} \pi_i^{-1}(H_i^{(\ell)}).
\]
Then by Lemma \ref{boundaryintersection}(\ref{boundarydiv}),
\[
	0 = \divr (\chi^u|_X) = \sum_{j=1}^m \left( \sum_{i=1}^n h_i a_{ij} \right) D_j + \sum_{i=1}^n \sum_{\ell=1}^{r_i} h_i^{(\ell)} E_i^{(\ell)}.
\]
Thus for all $j \in \{1, \dots, m\}$,
\[
	\sum_{i=1}^n h_i a_{ij} = 0,
\]
and for all $i \in \{1, \dots, n\}, \ell \in \{1, \dots, r_i\}$,
\[
	h_i^{(\ell)} = 0.
\]
Then by Lemma \ref{matrixentries}(\ref{matrixentrieslinearindependence}), for all $i \in \{1, \dots, n\}$,
\[
	h_i = 0.
\]
Therefore $\divr(\chi^u) = 0$, so $\chi^u \in \cO_\bP(\bP)^\times = k^\times$ and thus $u = 0$.
\end{proof}

\begin{proposition}
\label{properisgtropsurjective}
If $\cL'_1, \dots, \cL'_m$ are independent in $\Pic(X)$, then the map $M \to \cO_U(U)^\times/k^\times$ induced by $U \hookrightarrow T$ is surjective.
\end{proposition}

\begin{proof}
Let $g \in \cO_U(U)^\times$. Considering $g$ as a rational function on $X$, write
\[
	\divr g = d_1D_1 + \dots + d_m D_m + \sum_{i=1}^n \sum_{\ell=1}^{r_i} e_i^{(\ell)} E_i^{(\ell)}.
\]
Then
\begin{align*}
	\cO_X &\cong \cO_X(d_1D_1 + \dots + d_m D_m + \sum_{i=1}^n \sum_{\ell=1}^{r_i} e_i^{(\ell)} E_i^{(\ell)})\\
	&\cong \bigotimes_{j=1}^m \cL_j^{\otimes d_j + \sum_{i=1}^n \left( \sum_{\ell=1}^{r_i} e_i^{(\ell)} \right) a_{ij}}.
\end{align*}
Then because $\cL_1, \dots, \cL_n$ are independent in $\Pic(X)$, for each $j \in \{1, \dots, m\}$,
\[
	d_j + \sum_{i=1}^n \left( \sum_{\ell =1}^{r_i} e_i^{(\ell)} \right) a_{ij} = 0.
\]
Now consider the divisor on $\bP$
\[
	H = \sum_{i=1}^n \left( -\sum_{\ell =1}^{r_i} e_i^{(\ell)} \right) \pi_i^{-1}(H_i) + \sum_{i=1}^n \sum_{\ell=1}^{r_i} e_i^{(\ell)} \pi_i^{-1}(H_i^{(\ell)}).
\]
Then
\[
	\cO_{\bP}(H) \cong \bigotimes_{i=1}^n (\pi_i^* \cO_{\bP^{r_i}}(1))^{\otimes -\sum_{\ell=1}^{r_i} e_i^{(\ell)} + \sum_{\ell=1}^{r_i} e_i^{(\ell)}} \cong \cO_\bP.
\]
Thus there exists $u \in M$ such that $\divr (\chi^u) = H$. By Lemma \ref{boundaryintersection}(\ref{boundarydiv})
\begin{align*}
	\divr(\chi^u|_X) &= \sum_{j=1}^m \left( -  \sum_{i=1}^n \left( \sum_{\ell =1}^{r_i} e_i^{(\ell)} \right) a_{ij} \right) D_j + \sum_{i=1}^n \sum_{\ell=1}^{r_i} e_i^{(\ell)} E_i^{(\ell)}\\
	&= d_1D_1 + \dots + d_m D_m + \sum_{i=1}^n \sum_{\ell=1}^{r_i} e_i^{(\ell)} E_i^{(\ell)}\\
	&= \divr g.
\end{align*}
Thus $\chi^u|_U$ and $g$ have the same image in $\cO_U(U)^\times / k^\times$, so the map $M \to \cO_U(U)^\times/k^\times$ is surjective.
\end{proof}

\section{Immersive Geometric Tropicalization}
\label{immersivegeometrictropicalizationsection}

\subsection{Surfaces without Immersive Geometric Tropicalization}

In this sub-section, we will prove Theorem \ref{immersivegtrop}(\ref{surfacewithoutimmersivegtrop}).

\begin{lemma}
\label{pointblowupsurface}
Let $X$ be a smooth projective surface, let $U \subset X$ be a very affine open subvariety such that $X \setminus U$ is a simple normal crossing divisor, let $U \hookrightarrow T$ be a closed immersion into an algebraic torus, and let $\pi: X' \to X$ be the blow-up of $X$ at a point in $X \setminus U$.

If $U \subset X$ does not have immersive geometric tropicalization with respect to $U \hookrightarrow T$, then $U \subset X'$ does not have immersive geometric tropicalization with respect to $U \hookrightarrow T$.
\end{lemma}

\begin{proof}
Let $M$ be the character lattice of $T$, let $D_1, \dots, D_m$ be the irreducible components of $X \setminus U$, let $p \in X \setminus U$ be the center of $\pi: X' \to X$, let $D'_1, \dots, D'_m$ be the strict transforms of $D_1, \dots, D_m$, respectively, and let $E$ be the exceptional divisor of $\pi: X' \to X$.

By Theorem \ref{gtropproper}(\ref{gtropisproper}), $\val_{D_j}|_M \neq 0$ for all $j \in \{1, \dots, m\}$, so because $U \subset X$ does not have immersive geometric tropicalization with respect to $U \hookrightarrow T$, there exist $i \neq j \in \{1, \dots, m\}$ such that $D_i \cap D_j \neq \emptyset$ and $\val_{D_i}|_M$ and $\val_{D_j}|_M$ are linearly dependent.

If $p  \notin D_i \cap D_j$, then $D'_i \cap D'_j \neq \emptyset$ and $\val_{D'_i}|_M$ and $\val_{D'_j}|_M$ are linearly dependent. If $p \in D_i \cap D_j$, then $D'_i \cap E \neq \emptyset$ and $\val_{D'_i}|_M$ and $\val_E|_M$ are linearly dependent. In either case, $U \subset X'$ does not have immersive geometric tropicalization with respect to $U \hookrightarrow T$.
\end{proof}

\begin{lemma}
\label{nonruledminimalsurface}
Let $X$ be a smooth projective surface, let $U \subset X$ be a very affine open subvariety such that $X \setminus U$ is a simple normal crossing divisor, let $U \hookrightarrow T$ be a closed immersion into an algebraic torus, and let $U \subset \widetilde{X}$ be an open immersion into a smooth projective surface such that $\widetilde{X} \setminus U$ is a simple normal crossing divisor.

If $X$ is a non-ruled minimal surface and $U \subset X$ does not have immersive geometric tropicalization with respect to $U \hookrightarrow T$, then $U \subset \widetilde{X}$ does not have immersive geometric tropicalization with respect to $U \hookrightarrow T$.
\end{lemma}

\begin{proof}
By Lemma \ref{pointblowupsurface}, it suffices to show the existence of a map $\pi: \widetilde{X} \to X$ such that $\pi$ restricts to the identity on $U$ and such that $\pi$ is isomorphic to a sequence of point blow-ups.

There exists $\pi': \widetilde{X} \to X'$ such that $X'$ is a minimal surface and $\pi'$ is a sequence of point blow-ups. Let $U' \subset U$ be a nonempty open subvariety such that $\pi'|_{U'}: U' \to \pi'(U')$ is an isomorphism. Let $b: X' \darrow X$ be the birational map defined by the morphism $\pi'(U') \to X$ obtained by composing $(\pi'|_{U'})^{-1}: \pi'(U') \to U'$ with the inclusion $U' \subset X$. Then because $X'$ and $X$ are non-ruled minimal surfaces, $b$ is an isomorphism $X' \xrightarrow{\sim} X$. Set
\[
	\pi = b \circ \pi': \widetilde{X} \to X.
\]
Then $\pi$ restricts to the identity on $U$ and is isomorphic to a sequence of point blow-ups, and we are done.
\end{proof}

We now complete the proof of Theorem \ref{immersivegtrop}(\ref{surfacewithoutimmersivegtrop}).

\begin{proof}
Let $X'$ be a non-ruled minimal surface with Picard rank at least 2. For example, we can take $X'$ to be the Fermat Quartic. Set $m = 2$, $L = \Z$, and $\varphi: \Z^m \to L$ to be the map of lattices taking both standard basis vectors to $1 \in \Z$. Then $\varphi_\R|_{\R_{\geq 0}^m}: \R_{\geq 0}^m \to L_\R$ is proper as a map of topological spaces, so by Theorem \ref{gtropproper}(\ref{properisgtrop}), there exists an open subvariety $U \subset X'$ and a closed immersion $U \hookrightarrow T'$ into an algebraic torus such that $X' \setminus U$ is a simple normal crossing divisor and $U \subset X'$ does not have immersive geometric tropicalization with respect to $U \hookrightarrow T'$. Furthermore, because $X'$ has Picard rank at least 2, we can choose $U$ and $U \hookrightarrow T'$ so that $T'$ has character lattice $\cO_U(U)^\times / k^\times$ and $U \hookrightarrow T'$ is induced by a section $\cO_U(U)^\times / k^\times \to \cO_U(U)^\times$ of $\cO_U(U)^\times \to \cO_U(U)^\times / k^\times$.

Now let $U \hookrightarrow T$ be a closed immersion into an algebraic torus and $U \subset X$ be an open immersion into a smooth complete variety $X$ such that $X \setminus U$ is a simple normal crossing divisor. Then there exists a monomial morphism $T' \to T$ whose restriction to $U$ is the closed immersion $U \hookrightarrow T$. Therefore, $U \subset X$ does not have immersive geometric tropicalization with respect to $U \hookrightarrow T$ if $U \subset X$ does not have immersive geometric tropicalization with respect to $U \hookrightarrow T'$. But the latter holds by Lemma \ref{nonruledminimalsurface}.
\end{proof}

\subsection{Immersive Geometric Tropicalization of Non-Sch\"{o}n Varieties}

In this sub-section, we will prove Theorem \ref{immersivegtrop}(\ref{immersivenotschon}). First we recall a well known fact, which is straightforward to prove.

\begin{proposition}
\label{multiplicationfibers}

Let $T$ be an algebraic torus, let $U \hookrightarrow T$ be a closed subvariety, let $X(\Delta)$ be a $T$-toric variety, let $X$ be the closure of $U$ in $X(\Delta)$, and let
\[
	m: T \times X \to X(\Delta)
\]
be the multiplication map. If $Y$ is a torus orbit of $X(\Delta)$ and $y \in Y(k)$, then
\[
	m^{-1}(y) \cong \bG_m^{\dim T - \dim Y} \times (X \cap Y),
\]
where $X \cap Y$ is a scheme-theoretic intersection.
\end{proposition}

Let $X$ be a smooth projective variety of dimension at least 2. Let $\cL$ be a very ample line bundle on $X$. By Bertini's theorem, there exists $f \in \HH^0(\cL)$ such that $\divr f = D$ is smooth and irreducible. Set
\[
	\cE = \cL^{\otimes 2}, \qquad s= f^{\otimes 2} \in \HH^0(\cE), \qquad r = \dim \HH^0(\cE) -1.
\]
By Bertini's theorem, there exists $(s^{(\ell)})_{\ell \in \{1, \dots, r\}}$ such that
\begin{itemize}

\item each $\divr s^{(\ell)} = E^{(\ell)}$ is smooth and irreducible,

\item $D+E^{(1)} + \dots + E^{(r)}$ is a simple normal crossing divisor,

\item and $s, s^{(1)}, \dots, s^{(r)}$ is a basis for $\HH^0(\cE)$.

\end{itemize}

Let $x, x^{(1)}, \dots, x^{(r)}$ be homogeneous coordinates for $\bP^r$ with corresponding hyperplanes $H, H^{(1)}, \dots, H^{(r)}$ and let $X \hookrightarrow \bP^r$ be the closed immersion such that $x, x^{(1)}, \dots x^{(r)}$ pull back to $s, s^{(1)}, \dots, s^{(r)}$, respectively. Set
\[
	T = \bP^r \setminus \left( H \cup H^{(1)} \cup \dots \cup H^{(r)} \right)
\]
and
\[
	U = X \cap T \subset X.
\]
Let $M$ be the character lattice of the algebraic torus $T$, and let $N = M^\vee$. Let $w, w^{(1)}, \dots, w^{(r)} \in N$ be the first lattice points of the rays corresponding to the hyperplanes $H, H^{(1)}, \dots, H^{(r)}$, respectively.

\begin{lemma}
\label{immersivenotschonboundary}

\begin{enumerate}[(a)]

\item \label{immersivenotschonboundaryintersection} We have the scheme theoretic intersections
\[
	X \cap H = \divr s = 2D
\]
and for each $\ell \in \{1, \dots, r\}$,
\[
	X \cap H^{(\ell)} = \divr s^{(\ell)} = E^{(\ell)}.
\]

\item \label{immersivenotschonboundaryval} We have
\[
	\val_D|_M = 2w,
\]
and for each $\ell \in \{1, \dots, r\}$,
\[
	\val_{E^{(\ell)}}|_M = w^{(\ell)}.
\]

\end{enumerate}
\end{lemma}

\begin{proof}
This follows from the construction of the map $X \hookrightarrow \bP^r$.
\end{proof}

\begin{proposition}
$X \setminus U$ is a simple normal crossing divisor and $U \subset X$ has immersive geometric tropicalization with respect to $U \hookrightarrow T$.
\end{proposition}

\begin{proof}
By Lemma \ref{immersivenotschonboundary}(\ref{immersivenotschonboundaryintersection}),
\[
	X \setminus U = D + E^{(1)} + \dots + E^{(r)},
\]
so it is a simple normal crossing divisor. The remainder of the proposition follows from Lemma \ref{immersivenotschonboundary}(\ref{immersivenotschonboundaryval}) and the fact that any strict subset of the lattice points $\{w, w^{(1)}, \dots, w^{(r)}\}$ forms part of a basis for $N$.
\end{proof}

\begin{proposition}
The map $M \to \cO_U(U)^\times/k^\times$ induced by $U \hookrightarrow T$ is an isomorphism.
\end{proposition}

\begin{proof}
By setting $m = 1$, $L = \Z$, and $\varphi: \Z^m \to L$ to be multiplication by 2, we see that the construction of $U \hookrightarrow T$ is a special case of the construction used in the proof of Theorem \ref{gtropproper}(\ref{properisgtrop}). Thus the result follows from Propositions \ref{properisgtropinjective} and \ref{properisgtropsurjective}.
\end{proof}

The following proposition completes the proof of Theorem \ref{immersivegtrop}(\ref{immersivenotschon}).

\begin{proposition}
$U$ is not sch\"{o}n in $T$.
\end{proposition}

\begin{proof}
Let $\cC(X \setminus U)$ be the set of irreducible components of $X \setminus U$. For each collection $\cS \subset \cC(X \setminus U)$, set
\[
	\sigma_\cS = \pos(\val_E|_M \, | \, E \in \cS) \subset N_\R,
\]
and set
\[
	\Delta = \{\sigma_\cS \, | \, \cS \subset \cC(X \setminus U), \bigcap_{E \in \cS} E \neq \emptyset\}.
\]
Then by Lemma \ref{immersivenotschonboundary}(\ref{immersivenotschonboundaryval}), $\Delta$ is a sub-fan of the unimodular fan defining $\bP^r$ as a $T$-toric variety. Let $X(\Delta) \subset \bP^r$ be the $T$-toric variety defined by $\Delta$.

We first show that $\Delta$ is a tropical fan for $U$ giving $X \subset X(\Delta)$ as a tropical compactification of $U$. By construction, $X$ is the closure of $U$ in $X(\Delta)$. Now consider the multiplication map
\[
	m: T \times X \to X(\Delta).
\]
We need to show that $m$ is flat and surjective. By construction of $\Delta$, the fact that $D+E^{(1)}+\dots+E^{(r)}$ is a simple normal crossing divisor, and Lemma \ref{immersivenotschonboundary}(\ref{immersivenotschonboundaryintersection}) and Proposition \ref{multiplicationfibers}, the fibers of $m$ are nonempty and equidimensional. Thus because $T \times X$ and $X(\Delta)$ are smooth, this implies that $m$ is flat and surjective.

But again by Lemma \ref{immersivenotschonboundary}(\ref{immersivenotschonboundaryintersection}) and Proposition \ref{multiplicationfibers}, $m$ has a non-reduced fiber and thus is not smooth. Therefore $U$ is not sch\"{o}n in $T$.
\end{proof}

\bibliography{MVFT}{}
\bibliographystyle{plain}

\end{document}